\documentclass{article}

\usepackage{amssymb}
\usepackage{amsmath}
\usepackage{amsthm}
\usepackage{epsfig}
\usepackage{mathrsfs}
\usepackage{graphicx}
\usepackage{dsfont}
\usepackage{epsfig,color}

\newtheorem{theorem}{Theorem}[section]
\newtheorem{lemma}[theorem]{Lemma}
\newtheorem{proposition}[theorem]{Proposition}
\newtheorem{definition}[theorem]{Definition}
\newtheorem{remark}[theorem]{Remark}
\newtheorem{corollary}[theorem]{Corollary}

\def\R{{\mathbb R}}

\def\to{\rightarrow}
\def\tto{\longrightarrow}
\newcommand{\real}{\mathbb{R}}
\newcommand{\prob}{\mathcal{P}}
\newcommand{\B}{\mathcal{B}}

\newcommand{\abs}[1]{\left|#1\right|}
\newcommand{\dive}{\mathrm{div}}

\def\BV{{\mathrm {BV}}}

\def\eps{\varepsilon}
\def\vphi{\varphi}

\def\na{\nabla}
\def\pa{\partial}
\def\ds{\displaystyle}

\def\sgn{{\mathrm{sign}}}

\def\P{{\mathcal P}}
\def\supp{\mbox{supp}}

\title{Regularity of local minimizers of the interaction energy via obstacle
problems}

\author{J. A. Carrillo\thanks{Department of Mathematics, Imperial College
London, London SW7 2AZ, UK.}, M. G. Delgadino\thanks{University of Maryland,
USA.}, A. Mellet\thanks{University of Maryland, USA.}}

\begin{document}

\maketitle

\begin{abstract}
The repulsion strength at the origin for repulsive/attractive potentials determines the regularity of local minimizers of the interaction energy. In this paper, we show that if this repulsion is like Newtonian or more singular than Newtonian (but still locally integrable), then the local minimizers must be locally bounded densities 
(and even continuous for more singular than Newtonian repulsion). 
We prove this (and some other regularity results) by first showing that the potential function associated to a local minimizer solves an obstacle problem and then by using classical regularity results for such problems.
\end{abstract}

\section{Introduction}
Given a pointwise defined function $W:\R^N \to (-\infty,+\infty]$, we define the
interaction energy 
of a probability measure $\mu\in \P(\R^N)$ by
\begin{equation}\label{energy}
E[\mu]:=\frac{1}{2} \int_{\R^N}\int_{\R^N}W(x-y)d\mu(x)d\mu(y).
\end{equation}
Here, $ \P(\R^N) $ denotes the space of Borel probability measures, and throughout
the paper, we will always assume that the interaction potential $W$ satisfies
\begin{enumerate}
\item[(H1)] $W$ is a non-negative lower semi-continuous function in
$L^1_{loc}(\R^N)$.
\end{enumerate}
Under this assumption, the energy $E[\mu]$ is well defined for all $\mu\in
\P(\R^N)$, with $E[\mu]\in [0,+\infty]$. Local integrability of the potential
avoids too singular potentials for which the interaction energy is 
infinite for many smooth densities. These very singular potentials lead to very
interesting questions in crystallization \cite{T}, whose study is outside the
scope of this work. Furthermore,  under the assumption {\rm (H1)}, the potential
function $\psi$ associated to a given measure $\mu$:
$$ 
\psi  (x):=W*\mu(x) = \int_{\R^N} W(x-y) d\mu(y)
$$
can be defined pointwise in $\R^N$, and a simple application of Fatou's lemma
implies that $\psi $ is a lower semi-continuous function, see \cite[Lemma
2]{BCLR2}.

The goal of this work is to investigate the regularity properties of the local
minimizers of the interaction energy \eqref{energy}. 
For this, the keystone of this paper will be to show that  the potential function $\psi(x)$ associated to a local minimizer solves an obstacle problem. This fact  was suggested by the Euler-Lagrange conditions derived in  \cite{BCLR2} and will be made rigorous here in Section 3.
Note that in order to define 
precisely the notion of  local minimizers, we need to specify a topology on the set of probability measure.
We will use here the framework developed in  \cite{BCLR2}, where the authors consider  local minimizers of the energy
\eqref{energy} with respect to the optimal transport distance $d_\infty$. 
We refer to Section 2 for the main definitions and technicalities associated to the transport distance 
and a brief presentation of the main results of   \cite{BCLR2}.

Lots of numerical results
\cite{soccerball,FHK,FH,BCLR,BCLR2,BCY,Albietal,CHM,CCH2} show the rich
structure and variety of local/global minimizers of the interaction energy by
using different numerical approaches such as particle approximations, DG schemes
for the gradient flow equation associated to the energy \eqref{energy}, direct
resolution of the associated steady equations, radial coordinates, and so on.
The interaction potentials used in most of these numerical experiments are
repulsive near the origin and attractive at large distances. Typical choices are
radial potentials with a unique minimum $L$ for $r>0$, decreasing (repulsive)
before and increasing (attractive) after. In particular, for a system of two
identical particles, the discrete energy would then be minimized when they are located at
distance $L$ from each other. Particular relevant examples are Morse
potentials \cite{Dorsogna,BT2,KCBFL} and power-laws \cite{FHK,BCLR,CFT}.

These repulsive/attractive interaction potentials emanated from applications in
self-similar solutions for granular media models
\cite{LT,Carrillo-McCann-Villani03,Carrillo-McCann-Villani06}, collective
behavior of animals (swarming)
\cite{Mogilner2,Mogilner2003,Dorsogna,BT2,FH,FHK,KCBFL}, and self-assembly of
nanoparticles \cite{Wales1995,Rechtsman2010,Wales2010,Viral_Capside}. Let us
mention that local minimizers of the interaction energy can be seen as steady
states of the aggregation equation that have been studied thoroughly for fully
attractive potentials \cite{BCL,CDFLS} and repulsive/attractive potentials
\cite{FellnerRaoul1,FellnerRaoul2,Raoul,FHK,FH,BCLR,BCLR2,CFP,CDFLS2,BCY},
analysing qualitative properties of the evolution in different cases: finite
time blow-up, stabilization towards equilibria, confinement of solutions and so
on.

The beautiful result shown in \cite{BCLR2}, corroborated by the cited numerical
studies, is that the support of local minimizers of the interaction energy
increases as the repulsion at the origin gets stronger. In other words,
concentration of particles is not allowed on small dimensional sets when the
repulsion is large enough. Geometric measure theory techniques \cite{Mati} were
crucial to get the estimate on the dimension of the support based on the Euler-Lagrange conditions for local minimizers in transport distances. 

In this work, we concentrate on showing the regularity of the local minimizers
for repulsive/attractive potentials behaving at the origin like
\begin{equation}\label{eq:pot0} 
W(x)\sim \frac{1}{|x|^{N-2s}} , \qquad \mbox{ as } x\to 0, \mbox{ for some $s\in
(0,1]$ and $N\geq 2$}\,,
\end{equation}
with the understanding that $W(x)\sim -\log|x|$ if $s=1$ and $N=2$, and smooth
enough outside the origin. More precise statements will be given below and in
Section 3. In other words, we will assume that the repulsion at zero is stronger or equal to
Newtonian repulsion for dimension $N\geq 2$ but $W$ is still locally
integrable. Let us know make a summary of the particular results known in the literature.

As mentioned before, the case $s=1$ is of particular interest. It corresponds to
 Newtonian repulsion and it has received considerable attention due to its
various applications. A repetitively rediscovered result in this classical case
is that the global minimizer of the interaction energy for the potential
$$
W (x)=\frac{1}{|x|^{N-2}} + \frac{|x|^2}{2}\,, 
$$
is the characteristic function of a suitably chosen euclidean ball. This classical
result, using potential theory and capacities, was proved by Frostman \cite{Fr} (but in a 
bounded domain instead of confinement by quadratic potentials), and it has
connections with the eigenvalue distribution of random matrices \cite{PH,CGZ}.
This precise result can be found for instance in \cite[Proposition 2.13]{L-G}. In \cite{BLL}, 
the authors show that the uniform distribution in a ball is the asymptotic behavior of the corresponding
gradient flow evolution.  The uniqueness of the global minimizer for more singular than Newtonian
repulsion, i.e.,
$$
W (x)=\frac{1}{|x|^{N-2s}} + \frac{|x|^2}{2}\,, 
$$
with $0<s<1$, was obtained by Caffarelli and V\'azquez via the connection to a classical obstacle problem in \cite{CVobs}, and this strategy was also used in \cite{SV} to treat again the case $s=1$ for the evolution problem as in \cite{BLL}. A generalization with external confinement potential and repulsion due to potentials more singular than Newtonian has also been recently considered in \cite{CGZ} because of its connection to random matrices. Let us finally mention that  the case of the potential 
$$
W (x)=\frac{1}{|x|^{N-2}} + \frac{|x|^a}{a}\,, 
$$
with $a>2$ or $2-d<a<2$ has been analysed in \cite{FHK,FH} showing the existence and uniqueness of compactly supported radial minimizers of the interaction energy. Moreover, they show that they are bounded and smooth functions inside their support. The boundedness and smoothness inside the support of radial compactly supported minimizers was also proved for the so-called Quasi-Morse potentials in \cite{CHM}. These Quasi-Morse potentials behave at the origin as Newtonian potentials while they exhibit similar properties to Morse potentials in terms of existence of compactly supported radial minimizers. This particular case allow for explicit computations leading to analytic expressions for these minimizers.

The main result of this paper is that for kernels satisfying \eqref{eq:pot0},
and under reasonable assumptions on $W_a(x)=W(x)- |x|^{2s-N}$, local minimizers
$\mu$ of the interaction energy \eqref{energy} are absolutely continuous with
respect to the Lebesgue measure, and their density function lies in 
$L^\infty_{loc}(\R^N)$ when $s=1$ and in $C^\alpha_{loc}(\R^N)$ when $s\in(0,1)$. Furthermore, we will show that for $s=1$ the density function is in $BV_{loc}(\R^N)$ and that the support of
these local minimizers is a set with locally finite perimeter. In particular, our results will apply to potentials of the form
$$
W (x)=\frac{1}{|x|^{N-2s}} + \frac{|x|^q}{q}, \qquad \mbox{ for } q>0.
$$
These results will be obtained by exploiting the connection between the Euler-Lagrange
conditions for local minimizers and classical obstacle problems  \cite{CF}. 

In fact, we will show that the potential functions of local minimizers are locally
solutions of some obstacle problems. It is by using the regularity theory for the solutions of such obstacle
problems \cite{C,C1} that we will derive our main results on the regularity of local
minimizers. Note that Newtonian repulsion ($s=1$) will lead to the classical obstacle problem with the Laplace operator, while stronger repulsion ($s\in(0,1)$) will lead to fractional obstacle problems (with fractional power of the Laplace operator) which have been more recently studied, in particular in \cite{S,CSS} (we will also use some  results of  \cite{CVobs,CV} where these obstacle problems arise in the study of fractional-diffusion versions of the porous medium equation).

It should be noted that potentials that are less repulsive than Newtonian (but still repulsive at the origin) also lead to some obstacle problems. However these involve elliptic operators of higher order. A good example is the case where $V(x)\sim-|x|$ in dimension $N=3$, which leads to a biharmonic obstacle problem. The regularity theory for these higher order obstacle problems is different, and far less developed, than that of the problems considered here. It is for that reason that they will not be discussed in this paper, but will be the object of a forthcoming work.

Let us finally comment that some of our results will require some additional
uniformity assumptions on the potential $W_a$ at infinity if the support of the
local minimizer is not compact. In fact, the existence of compactly supported
global miminizers for the interaction energy is a very interesting question by
itself connected to statistical mechanics \cite{HStable}. This property has
recently been shown \cite{CCP,SST} under natural
conditions on the interaction potential $W$ related to non $H$-stability as defined
in \cite{HStable,Dorsogna}.

The plan of this paper is as follows. Section 2 is devoted to describing
the notion of local minimizers used in this paper and to recalling briefly
the results of \cite{BCLR2} that are  relevant to our study. 
Section 3 gives the precise statements of the main
results of this paper both for the Newtonian and more singular than Newtonian
cases. The aims of Sections 4 and 5 are to show the main results in the
Newtonian repulsion case. Section 6 is devoted to the case of more
singular than Newtonian repulsions. The final section is devoted to obtain
uniqueness results in the particular case in which the attraction is a quadratic potential.
A couple of technical results (mean value formulas) are presented in Appendices.


\section{Local minimizers \& Euler-Lagrange Conditions}
\setcounter{equation}{0}

In this section, we give a brief summary of the main results obtained in \cite{BCLR2}. 
First, in order to define our notion of  local minimizers of $E[\mu]$, we have
to give precise definitions of the topology that we will use in the space of probability
measures. We recall that $\prob(\real^N)$ is  the set of Borel probability measures on
$\real^N$ and we denote by  $\B(\real^N)$ the family of Borel subsets of $\real^N$. The
support of a measure $\mu \in \prob(\real^N)$ is the closed set
defined by
\begin{equation*} 
\supp (\mu):= \{x \in \real^N: \mu(B(x,\epsilon))>0  \text{ for
all } \epsilon>0 \}\,.
\end{equation*}
A family of distances between probability measures have been classically
introduced by means of optimal transport theory, we will review briefly some of
these concepts, we refer to \cite{GS,V} for further details. A probability
measure $\pi$ on the product space $\real^N \times \real^N$ is said to be a
transference plan between $\mu \in \prob(\real^N)$ and $\nu \in \prob(\real^N)$ 
if
\begin{equation}\label{marginal}
\pi(A \times \real^N)=\mu(A) \quad \text{and} \quad \pi(\real^N
\times A)=\nu(A)
\end{equation}
for all $A \in \B(\real^N)$. If $\mu, \nu \in \prob(\real^N)$,
then
$$
\Pi(\mu,\nu):=\{ \pi \in \prob(\real^N \times \real^N):
\eqref{marginal} \text{ holds for all }A \in \B(\real^N)\}
$$
denotes the set of admissible transference plans between $\mu$ and
$\nu$. Informally, if $\pi \in \Pi(\mu,\nu)$ then $d \pi(x,y)$
measures the amount of mass transferred from location $x$ to
location $y$. Observe that
$\sup_{(x,y)\in \supp (\pi) } \abs{x-y}$ represents the maximum
mass displacement when transporting $\mu$ onto $\nu$ 
by the transference plan $\pi$. 
As in \cite{BCLR2}, we will consider 
local
minimizers
of the energy functional with respect to the $\infty$-Wasserstein
distance $d_\infty$.
This distance is defined as the optimal maximal mass displacement given
by
\begin{equation*}
    d_\infty(\mu,\nu) = \inf_{\pi \in \Pi(\mu,\nu)}  \sup_{(x,y)\in \supp (\pi)
}
    \abs{x-y},
\end{equation*}
which can take infinite values in general, but is obviously finite for
compactly supported measures. This distance induces a complete
metric structure restricted to the set of probability measure with
finite moments of all orders, $\prob_\infty(\real^N)$, as proven
in \cite{GS}. 
By considering local
minimizers  with respect to this $d_\infty$ distance, we are only allowing
 small perturbations in the $d_\infty$ sense (corresponding to mass being moved short distances). 
For such perturbations, the behavior of the energy functional
is close to particle-like approximations, see \cite{BCLR2,CCH} for more
discussion related 
to this interpretation. 

Furthermore, we note that the  $d_\infty$-topology is the coarsest topology among all the topologies induced by transport
distances. So our results will automatically hold for any local minimizers with respect to the $p$-Wasserstein distance for any $p\in[1,\infty]$.
We recall that for $1\le p < \infty$ the distance $d_p$ between two
measures $\mu$ and $\nu$ is defined by
\[
    d_p^p(\mu,\nu)=\inf_{\pi\in\Pi(\mu,\nu)}\left\lbrace \iint_{\R^N\times \R^N}
    |x-y|^pd\pi(x,y)\right\rbrace.
\]
Note that $d_p(\mu,\nu)<\infty$ for $\mu,\nu\in \prob_p(\real^N)$
the set of probability measures with finite moments of order $p$.
Since $d_p(\mu,\nu)$ is increasing as a function of $1\le p <
\infty$, one can show that it converges to $d_\infty(\mu,\nu)$ as
$p\to\infty$. Since the distances are ordered with respect to $p$,
it is obvious that the topologies are also ordered.
We now can give the following definition:
\begin{definition}\label{def:min}
We say that $\mu$ is an $\eps$-local minimizer (or simply $\eps$-minimizer) for
the energy $E$ with respect to $d_\infty$, if $E[\mu]<\infty$ and 
$$ 
E[\mu] \leq E[\nu]
$$
for all $\nu \in \P(\R^N)$ such that $d_\infty(\mu,\nu)< \eps$.
\end{definition}

Note that if $\mu$ is compactly supported, the previous definition is reduced to
check the local minimality for compactly supported perturbations $\nu$.
Minimizers should correspond to equilibrium configurations for the evolution
equation obtained by steepest descent of the energy. However, being a functional
on probability measures, the steepest descent has to be understood in the
Wasserstein sense as in \cite{Otto}. The transport distance to be used in that
case is $d_2$, and the steepest descent reads as
\begin{equation}\label{eq:main}
\frac{\partial \mu_t}{\partial t} = \dive \left[\left(\nabla
\frac{\delta E}{\delta\mu}\right)\mu_t\right] = \dive \left( \mu_t \nabla\psi_t
\right) \qquad x\in\R^N \, , t>0.
\end{equation}
For steady configurations, we expect $\na \psi = 0$ on the support of $\mu$, due
to the formal energy dissipation identity for solutions, i.e.,
$$
\frac{d}{dt} E[\mu_t] = -\int_{\R^N} |\nabla \psi_t|^2\,d\mu_t
$$
where $\mu_t$ is any solution at time $t$ of \eqref{eq:main} and $\psi_t=W\ast
\mu_t$ its associated potential. Therefore, the points on the support of a local
minimizer $\mu$ of the energy $E$ should correspond to critical points of its
associated potential $\psi$. This fact is made rigorous in:

\begin{proposition}[{\rm \cite[Proposition 1]{BCLR2}}]\label{prop:minsupp}
Assume that $W$ satisfies {\rm (H1)} and let $\mu$  be a $\eps$-minimizer of the
energy $E[\mu]$ in the sense of Definition {\rm \ref{def:min}}. Then any point
$x_0\in\supp(\mu)$ is a local minimimum of $\psi=W*\mu$ in the sense that
\begin{equation}\label{eq:psimin}
 \psi(x_0) \leq \psi(x) \mbox{ for a.e. } x\in B_\eps(x_0).
 \end{equation}
\end{proposition}

\begin{remark}\label{rem:epsilon}
An attentive reading of the proof of {\rm \cite[Proposition 1]{BCLR2}} leads to
the important observation that the $\varepsilon$ appearing in \eqref{eq:psimin}
is the same as the $\eps$ appearing in Definition {\rm\ref{def:min}}. In
particular, it is independent of the point $x_0$. Moreover, only local
integrability of the interaction potential is needed for that proof, i.e., there
is no need of uniform local integrability of $W$ for the proof in
{\rm\cite[Proposition 1]{BCLR2}}.
\end{remark}

Let us also point out that the result in Proposition \ref{prop:minsupp} for an
$\eps$-minimizer does not imply that its potential $\psi$ is constant in each
connected component of the support of $\mu$. Additional information is needed in
terms either of the regularity of the potential $\psi$ (continuity of $\psi$, see next section), or in
terms of the interaction potential $W$ itself (see \cite[Proposition 2]{BCLR2}).
Another possibility is to change the topology. In fact, as proven in
\cite[Theorem 4]{BCLR2}, if $\mu$ is a $d_2$-local minimizer of the energy, then
the potential $\psi$ satisfy the Euler-Lagrange conditions given by
\begin{enumerate}
\item[(i)] $\psi(x)=(W\ast \mu)(x)=2E[\mu]$ $\mu$-a.e.

\item[(ii)] $\psi(x)=(W\ast \mu)(x)\leq 2E[\mu]$ for all $x\in
\supp(\mu)$.

\item[(iii)] $\psi(x)=(W\ast\mu)(x)\geq 2E[\mu]$ for a.e. $x\in\R^N$.
\end{enumerate}
These conditions simplify to
\begin{align*}
\psi(x)=(W\ast \mu)(x) &= 2E[\mu]\quad \text{for a.e.}\quad x\in\supp(\mu)\,,\\
\psi(x)=(W\ast \mu)(x) &\geq 2E[\mu]\quad \text{for a.e.} \quad x\in
\R^N\setminus \supp(\mu)\,,
\end{align*}
if $\mu$ is absolutely continuous with respect to the Lebesgue measure. This
problem is already quite close to classical obstacle problems encountered in
semiconductors \cite{CF}. In fact, we leave to the interested reader to check
that the particular case
$$
W(x)=\frac{1}{|x|^{N-2}} + \frac{|x|^2}{2}
$$
with $N\geq 3$, coincides precisely with the problem solved in \cite{CF}, see
also \cite{CDMS}.

Now, let us connect the qualitative results in \cite{BCLR2} to the present work.
The main result in \cite{BCLR2} shows that the support of $\eps$-minimizers for
potentials that are $\beta$-repulsive at the origin gets larger and larger as the
repulsion gets stronger and stronger at the origin. To be more precise if the
potential satisfies
$$
-\Delta W(x)\geq \frac{C}{|x|^{\beta}} \, \mbox{ for } 0<\beta<N
$$
at the origin, see \cite[Definition 2]{BCLR2} for a more precise statement, then
any $\eps$-minimizer of $E$ has the property that the Hausdorff dimension of its
support is greater than or equal to $\beta$. We observe that the notion of
$\beta$-repulsivity with $2\leq \beta<N$ at the origin in \cite[Definition
2]{BCLR2} implies, in particular, that $W(0)=+\infty$ for singular at zero
potentials, avoiding trivial minimizers, see \cite[Subsection 3.3]{BCLR2}. In
this work, we will show that when the potential at the origin is even more
repulsive, i.e., it satisfies \eqref{eq:pot0} in a sense to be made precise in the
next section, then the $\eps$-minimizers must be very regular and smooth. This
regularity will emanate from the obstacle-like problem to which the
Euler-Lagrange conditions written above are equivalent to. Observe that when the
potential satisfies \eqref{eq:pot0}, we are saying in some sense that the
potential is $\beta$-repulsive with $N\leq \beta<N+2$.


\section{Main Results \& Strategy}
\setcounter{equation}{0}

In this section, we will first discuss in details the Newtonian case $s=1$ and then turn to the case $s\in(0,1)$. 
Let
us denote by $\omega_N$ the area of the $N$-dimensional unit sphere. We denote
by
$$
V(x)=\left\{\begin{array}{ll}
\ds \frac{1}{N(N-2)\omega_N}\frac{1}{|x|^{N-2}} & \mbox{ if $N\geq 3$}\\[10pt]
\ds \frac{1}{2\pi}\log\frac1{|x|} & \mbox{ if $N=2$}
\end{array}\right.
$$
the fundamental solution of Laplace equation, which satisfies
\begin{equation}\label{eq:delta} 
-\Delta V=\delta_0.
\end{equation}
Without loss of generality (we can always multiply $W$ by a constant without
changing the minimizers), we will assume that
\begin{enumerate}
\item[(H2)] The function  $W_a (x):= W(x)-V(x)$ satisfies:
\begin{equation}\label{eq:D2Wa}
       \Delta W_a\in L^p_{loc}(\R^N) \quad \mbox{ for some $p\in
(\frac{N}{2},\infty]$}
\end{equation}
and
\begin{equation}\label{eq:D1Wa}
 \Delta W_a (x) \geq -C_* \quad \mbox{ a.e. $x\in \R^N$} \,.
\end{equation}
\end{enumerate}

We use the notation $W_a$, because it is convenient to think that $W=V+W_a$
where $V$ describe the repulsive interactions, while $W_a$ describes the
attractive interactions. But of course $W_a$ could include both repulsive and
attractive effects. However, assumption {\rm(H2)} ensures that the dominant
behavior near $0$ is Newtonian repulsion. Without loss of generality, we will
assume that $W(0)=+\infty$ to avoid trivial local minimizers given by Dirac
Delta at a point. 
Note that Sobolev embeddings
theorems imply that $W_a$ is a H\"older continuous function in $\R^N$.

\begin{remark}
Hypothesis {\rm (H2)} is in particular satisfied  by the power potentials
$W_a(x)={|x|^q}/{q}$ with $q>0$. Indeed, we have $\Delta W_a = (q+N-2)
|x|^{q-2}$ which is bounded below (in fact, we have $ \Delta W_a\geq 0$ provided
$q>2-N$). Furthermore, we see that $\Delta W_a\in L^\infty_{loc}(\R^N)$ if
$q\geq2$, and  $\Delta W_a\in L^p_{loc}(\R^N)$ for all  $p<\frac{N}{2-q}$, if
$q<2$. In particular,  \eqref{eq:D2Wa} holds for all $q>0$.
\end{remark}

\medskip

When $\mu$ is not compactly supported, the fact that $\Delta W_a$ is
only locally in $L^p(\R^N)$ will be problematic, and we will need to assume that
$\Delta W_a$ is locally uniformly in $L^p(\R^N)$. 

More precisely, we always assume that one of the following
holds: either
\begin{enumerate}
\item[(H3a)] The support of $\mu$,  $\supp( \mu)$, is compact in $\R^N$
\end{enumerate}
or
\begin{enumerate}
\item[(H3b)] $\Delta W_a$ is locally uniformly in $L^p(\R^N)$ for some $p\in
(\frac N2 ,\infty]$, that is
there exists a constant $M$ such that  
\begin{equation}\label{eq:M}
||\Delta W_a||_{L^p(B_1(x))}\leq M \quad \mbox{  for all $x\in \R^N.$}
\end{equation}
Furthermore, in dimension $N=2$, we assume that $\mu$ is such that
\begin{equation}\label{eq:logmu}
\int_{\R^N} \log(1+ |x|)\, d\mu <\infty.
\end{equation}
\end{enumerate}

We note that {\rm (H3b)} holds typically for potential that do not grow too much at $\infty$, while it is   expected that for potentials that grow fast enough at $\infty$, local minimizers of the energy  have compact support, i.e.  {\rm (H3a)}  should hold (this last fact remains to be proved though). 
So conditions {\rm (H3a)}  and {\rm (H3b)} should be seen as complementary.
We recall also that  the existence of compactly supported global minimizers of the
interaction energy $E$ has recently been proved in \cite{CCP,SST} under natural
conditions on the interaction potential related to non $H$-stability as defined
in \cite{HStable,Dorsogna}. 
Thus, relevant minimizers, in applications such as
swarming \cite{Dorsogna,BT2,FHK,FH,BCLR,BCLR2,CHM}, are typically compactly
supported. 

When $\supp(\mu)$ is compact, say $\supp(\mu)\subset B_R(0)$, we can cut-off the
kernel $W$ in a smooth way outside the ball $B_{4R}(0)$. The density  $\mu$ will
still be an $\epsilon$-minimizer of the energy $E$ and its potential $\psi$ will
be unchanged in the ball $B_{2R}(0)$. 
So whenever assuming {\rm (H3a)}, it is possible to replace \eqref{eq:D2Wa} with
$\Delta W_a \in L^p(\R^N)$. In conclusion, whether {\rm (H3a)} or {\rm (H3b)} is
satisfied, we can always assume that \eqref{eq:M} holds.
\medskip

We now have all the assumptions that we will need on  the interaction potential in the case of 
Newtonian repulsion. We will later see how those conditions must be changed for more repulsive potential (see Section \ref{sec:more}). But first,  let us give the main results in this Newtonian case.
 

\subsection{An obstacle problem}
As mentioned earlier, the keystone of this paper is the observation that the
potential function $\psi $ solves (locally) an obstacle problem. 
In order to make this fact rigorous, we will first need to prove that $\psi$ is
a continuous function. Our first result is thus the following proposition:

\begin{proposition}[Continuity of the potential]\label{prop:continuity}
Let $\mu$ be an $\eps$-minimizer of $E$  in the sense of Definition {\rm
\ref{def:min}}, and assume that 
{\rm (H1)}, {\rm (H2)}, and either {\rm (H3a)} or {\rm (H3b)} hold. Then the
potential $\psi(x):=W*\mu (x)$ associated to $\mu$ is a continuous function in
$\R^N$.
\end{proposition}

This proposition will be proved in Section \ref{sec:proofcont}. As a
consequence, we can now show that $\psi$ is locally the solution of an obstacle
problem in the neighborhood of any point in the support of $\mu$. Indeed, under
the conditions {\rm (H1)}-{\rm (H2)}-{\rm (H3x)}, using both Propositions
\ref{prop:minsupp} and \ref{prop:continuity}, we see that for any point
$x_0\in\supp(\mu)$
\begin{equation*}
 \psi(x)\geq \psi(x_0) \mbox{ for all } x\in B_\eps(x_0)
\end{equation*}
holds (see Corollary \ref{cor:minsupp}). Furthermore, since $\eps$
does not depend on $x_0$ (see Remark \ref{rem:epsilon}), this implies
\begin{equation}\label{eq:psiminequal}
\psi (x) = \psi(x_0) \quad  \mbox{ in } B_\eps(x_0) \cap\supp(\mu).
\end{equation}
Next, we observe that \eqref{eq:delta} implies
$$
-\Delta \psi = \mu -\Delta W_a * \mu \quad \mbox{ in } \mathcal D'(\R^N)
$$
where (using {\rm (H3)} and Minkowski's integral inequality), $\Delta W_a *
\mu\in L^p_{loc}(\R^N)$. In particular, since $\mu$ is a non-negative measure,
we deduce
$$ 
-\Delta \psi \geq  -\Delta W_a * \mu \quad \mbox{ in $B_\eps(x_0)$}.
$$
Furthermore, if $x\in B_\eps(x_0)$ is such that $\psi(x) >\psi(x_0)$,
\eqref{eq:psiminequal} implies that $x\notin \supp(\mu)$, and so (by definition
of $\supp(\mu)$), $\mu(B_r(x))=0$ for some small $r>0$.
We deduce
$$ 
-\Delta \psi = -\Delta W_a * \mu \quad \mbox{ in } \mathcal D'(\R^N) \mbox{ in 
} B_\eps(x_0)\cap \{\psi>\psi(x_0)\}\,.
$$
We thus have the following proposition:

\begin{proposition}\label{prop:obstacle}
For all $x_0\in\supp (\mu)$,
the potential function $\psi$ is equal, in $B_\eps(x_0)$, to the unique solution
of the obstacle problem 
\begin{equation}\label{eq:obstacle2l}
\left\{
\begin{array}{rll}
\vphi &\geq C_0,\quad & \mbox{ in } B_\eps(x_0) \\
-\Delta \vphi &\geq  - F(x), \quad & \mbox{ in } B_\eps(x_0) \\
-\Delta\vphi &=  - F(x), & \mbox{ in } B_\eps(x_0)\cap\{\vphi >C_0\} \\
\vphi &= \psi, & \mbox{ on  } \pa B_\eps(x_0),
\end{array}
\right.
\end{equation} 
where $C_0=\psi(x_0)$ and $F(x)=\Delta W_a*\mu\in L^p_{loc}(\R^N)$. Furthermore,
the measure $\mu$ is given by
\begin{equation}\label{eq:mupsil} 
\mu = -\Delta \psi + F  .
\end{equation}
\end{proposition}

We just showed that $\psi$ solves \eqref{eq:obstacle2l}. For the uniqueness, we
refer to \cite{KS}. Since $F$ depends on $\mu$ itself, it seems difficult
to exploit \eqref{eq:obstacle2l} to identify local minimizers or to prove global
properties such as uniqueness or radial symmetry. However, because $F$ is more
regular than $\mu$, we will be able to use \eqref{eq:obstacle2l} to study the
regularity of these local minimizers.  

We also insist here on the fact that in general the constant $C_0$ might depend
on the choice of $x_0\in \supp(\mu)$. For global minimizers, as well as for
local $d_2$ minimizers, this constant is actually independent of $x_0$ as
discussed in the previous section, see \cite[Theorem 4]{BCLR2}. 

The equation \eqref{eq:mupsil} suggests that there is a relation between the
support of $\mu$ and the coincidence set $\psi=\psi(x_0)$.
In fact,  it is easy to check that $ \mu=0 $ in the open set $\{\psi
>\psi(x_0)\}\cap B_\eps(x_0)$ in the sense that $\mu (\{\psi >\psi(x_0)\}\cap
B_\eps(x_0))=0$. We thus deduce using the continuity of $\psi$ that 
$$ 
\supp(\mu) \cap B_\eps(x_0) \subset \{\psi=\psi(x_0)\}\cap B_\eps(x_0).
$$
But it is not obvious that these two sets should be equal. Nevertheless, we
shall later see that, under a non-degeneracy condition on $F$,  they are equal
up to a set of measure zero.


\subsection{Regularity of $\psi$ and $\mu$}
Proposition \ref{prop:obstacle} will enable us to use classical regularity
results  for the obstacle problem to study the properties of $\eps$-minimizers
of $E$. Our first result is the following:

\begin{theorem}[$L^\infty$ regularity of $\mu$]\label{thm:regularity}
Assume $W$ satisfies {\rm (H1)} and {\rm (H2)}. Let $\mu$ be a compaclty
supported $\eps$-minimizer in the sense of Definition {\rm \ref{def:min}}.
Assume moreover that one of the followings hold: Either
\begin{itemize}
\item[(i)] {\rm (H2)} with $p>N$,

or
\item[(ii)] {\rm (H2)} with $\frac N 2<p\leq N$ and $\Delta W_a\in
W^{\eps,1}_{loc} (\R^N)$ for some small $\eps>0$. 
\end{itemize}
Then the potential function $\psi $ is in $\mathcal C^{1,1}(\R^N)$. In
particular, the measure 
$\mu$ is absolutely continuous with respect to the Lebesgue measure and there
exists a function 
$\rho \in L^\infty(\R^N)$ such that $\mu = \rho(x) d\mathcal L^N$. 
Finally, we have $\rho = \Delta W_a * \rho$ in the interior of $\supp(\mu)$.
\end{theorem}

Note that assumption $(ii)$ requires slightly more regularity, but a lot less
integrability that $(i)$. It is is satisfied by interaction potentials of the
form ${|x|^q}/{q}$ as soon as $q>2-N$.
The proof of this proposition will be given in Section 5. 

\begin{remark}
Under the conditions of Theorem {\rm \ref{thm:regularity}}, we can show that
local minimizers are actually stationary states of the aggregation equation
\eqref{eq:main}. Indeed, since $\nabla \psi \in C^{0,1}(\R^N)$ and $\rho \in
L^\infty(\R^N)$ we have  $\rho\nabla\psi \in L^\infty(\R^N)$. Moreover, since
$\nabla \psi=0$ in the interior of $\supp(\rho)$, then $\rho\nabla\psi=0$ a.e.
in $\R^N$, and thus $\rho$ is a classical stationary solution to
\eqref{eq:main}.
\end{remark}

Since $\Delta W_a \in L^p_{loc}(\R^N)$ and $\rho\in L^\infty(\R^N)$ with compact
support, a classical result for the convolution of functions implies that
$\Delta W_a * \rho$ is a continuous function in $\R^N$. In particular, we deduce
that $\rho$ is continuous in the interior of $\supp(\mu)$.  But, in general, we
will have $\Delta W_a * \rho\neq 0$ on $\pa (\supp(\mu))$ and so we expect
$\rho$ to be  discontinuous in $\R^N$. As a consequence, we cannot readily
obtain further regularity for $\rho$ without assuming more regularity for $W_a$.

Obviously, if $W_a$ is smooth in $\R^N$, then  $\rho$ will be smooth in the
interior of $\supp(\mu)$.
But  it is not very difficult to prove (using a bootstrapping argument)  that if
$\Delta W_a$ is smooth in $\R^N\setminus \{0\}$ (as is the case for power like
interaction potentials), then $\rho$ will be smooth in the interior of
$\supp(\mu)$.

Finally, under the assumptions of Theorem \ref{thm:regularity}, we note that since $\psi
\in W^{2,\infty}$, we have
$$ 
\Delta \psi = -\rho + \Delta W_a * \rho = 0 \quad \mbox{ a.e. in
}\{\psi=\psi(x_0)\}.
$$
If we assume further that $\Delta W_a * \rho >0$ in $B_\eps(x_0)$, then we have
$\rho(x)>0$ a.e. in $\{\psi=\psi(x_0)\}$ and thus
\begin{equation}\label{eq:suppcontact}
\mbox{meas} \left(\{\psi=\psi(x_0)\}\cap B_\eps(x_0)\setminus  \supp(\mu)\right)
= 0\,,
\end{equation}
in other words, the support of $\mu$ and the coincidence set
$\{\psi=\psi(x_0)\}$ are the same up to a set of measure zero.
\medskip 

As noted above, $\rho$ is expected to be a discontinuous function and so does  not belong, in general, to $W^{1,1}_{loc}$.
However, under appropriate regularity assumption on $\Delta W_a$, we can prove that
$\rho$ is in $\BV_{loc}(\R^N)$:

\begin{theorem}[Regularity of $\supp(\mu)$]\label{thm:BV}
Under the assumptions of Theorem \ref{thm:regularity}, assume further that 
$$ 
\Delta W_a \in  W^{1,1}_{loc}(\R^N) .
$$
Then the density $\rho$ lies in $\BV_{loc}(\R^N)$. Furthermore, if $\Delta W_a *
\rho >0$ in a neighborhood of $\pa(\supp(\mu))$, then $\supp(\mu)$ is a set
with locally finite perimeter.
\end{theorem}

Note that the condition that $\Delta W_a * \rho >0$ in a neighborhood of $\pa
(\supp(\mu))$ is in particular satisfied if $\Delta W_a(x) $ is non-negative for
all $x$ and not identically zero  (which is the case when $W_a(x)={|x|^q}/{q} $
with $q>2-N$). This condition implies that $\rho$ has a nonzero continuous
extension on $\pa (\supp(\mu))$ from the interior of $\supp(\mu)$. In
particular, $\rho$ has a jump
discontinuity at the boundary of its support, and the $BV$ regularity is thus optimal in that sense.

Finally, let us point out that  there are numerous results in the literature
concerning further regularity of the free boundary $\pa (\supp(\mu))$  for the
obstacle problem, always under the same non-degeneracy requirement that $\Delta
W_a * \mu >0$ in a neighborhood of the free boundary, see \cite{C,Blank}.
Clearly many of  these results could be used here, but we will not pursue this direction, as we are mainly interested in the regularity of the measure $\mu$  itself.
 


\subsection{More singular than Newtonian case}\label{sec:more}
Many of the results stated in the previous subsections can be extended without
too much difficulty to the case in which the repulsion at zero in more singular
than Newtonian but still locally integrable, that is, when
$$ 
W(x)\sim \frac{1}{|x|^p} \quad\mbox{ as } |x|\to 0, \mbox{
with $p\in (N-2,N)$}
$$
The general approach is the same as that described in the previous section, but
the obstacle problem \eqref{eq:obstacle2l} will be replaced by a fractional
obstacle problem of order less than $2$. Because such fractional obstacle
problem enjoys very good properties, we will be able to derive regularity
results for $\eps$-minimizers.

To make this more precise, we first recall that $(-\Delta)^s$ denotes the
fractional Laplace operator of order $s\in(0,1)$, which can be defined as a
singular integral operator, or, using Fourier transform, as the operator with
symbol $|\xi|^{2s}$. For $s\in(0,1)$, it is then well known that  the function
$$ 
V_s(x) =\frac{c_{N,s}}{|x|^{N-2s}} 
$$
is the fundamental solution of the fractional Laplace equation (for an
appropriate choice of the constant $c_{N,s}$). More precisely, it satisfies
$(-\Delta)^{s} V_s=\delta_0$. In this section, we fix $s\in(0,1)$ and consider
an interaction potential that satisfies $W(x)\sim V_s(x)$ for $s\to 0$, that is,
we replace hypothesis {\rm (H2)} with:
\begin{enumerate}
\item[(H2s)] The function $W_a (x):= W(x)-V_s(x)$ satisfies:
\begin{equation}\label{eq:fracLph}
 (-\Delta)^s W_a\in L^p_{loc}(\R^N) \;\; \mbox{ for some }\;
p\in(\frac{N}{2s},\infty].
\end{equation}
\begin{equation}\label{eq:fraclaplace}
 (-\Delta)^s W_a(x)\leq C_* \;\; \mbox{a.e.} \;\; x\in\R^N.
\end{equation}
\end{enumerate}
As before, we also define $W(0)=+\infty$.

Finally, because of the non-locality of the fractional laplacian, it is much more difficult to deal with non compactly supported minimizers in this framework. For the sake of simplicity, and because the minimizers of interest in most applications have this property,
we will thus only consider compactly supported local minimizers in this section (so {\rm (H3a)} holds).

\begin{remark}
We note that when considering power-laws interaction potentials
$W_a(x)=\frac{|x|^q}{q}$, we get
$$  
(-\Delta)^s W_a = |x|^{q-2s} \frac {C_{N,s}}{ q} \mbox{P.V.}\int_{\R^N}
\frac{1-|z|^q}{|1-z|^{N+2s}}\, dz\, ,
$$
where we used the singular integral formulation of the fractional Laplacian. In
particular, we need $q<2s$ in order for the integral to be convergent and $q>0$
in order for \eqref{eq:fracLph} to hold. However, the restriction $q<2s$ can be
eliminated by truncating the potential outside a large ball.
\end{remark}

As in the Newtonian case, the first step is to establish the continuity of the
potential $\psi$:

\begin{proposition}[Continuity of the potential]\label{prop:continuitys}
Assume that the interaction potential $W$ satisfies {\rm (H1)} and {\rm
(H2s)}, and let $\mu$ be an $\eps$-minimizer in the sense of Definition
{\rm \ref{def:min}} such that $\supp( \mu)$ is compact. Then the potential
$\psi(x):=W*\mu (x)$ is a continuous function in $\R^N$.
\end{proposition}

As a consequence, we can show (proceeding as in previous subsections):

\begin{proposition}\label{prop:obstaclefrac}
Under the conditions of Proposition {\rm \ref{prop:continuitys}}, for all
$x_0\in\supp (\mu)$, the potential function $\psi$ is equal, in $B_\eps(x_0)$,
to the unique solution of the obstacle problem 
\begin{equation*}
\left\{
\begin{array}{rll}
\vphi &\geq C_0,\quad & \mbox{ in } B_\eps(x_0) \\
(-\Delta)^s \vphi &\geq  - F(x), \quad & \mbox{ in } B_\eps(x_0) \\
(-\Delta)^s \vphi &=  - F(x), & \mbox{ in } B_\eps(x_0)\cap\{\vphi >C_0\} \\
\vphi &= \psi, & \mbox{ in  } \R^N\setminus B_\eps(x_0),
\end{array}
\right.
\end{equation*} 
where $C_0=\psi(x_0)$ and $F(x)=-(-\Delta)^s W_a*\mu\in L^p_{loc}(\R^N)$.
Furthermore, the density $\mu$ is given by
\begin{equation*}
\mu = (-\Delta)^s \psi + F  .
\end{equation*}
\end{proposition}

This type of obstacle problem has been studied by numerous authors in recent
years, in particular by L. Silvestre \cite{S}.
However, some aspects of the theory for this fractional obstacle problem are
different, or not as developed yet, as that of the regular obstacle problem. 
The only regularity result we will prove is that the density $\mu$ is H\"{o}lder
continuous:

\begin{theorem}[Regularity of $\mu$]\label{thm:regularityfrac}
Assume that $W$ satisfies {\rm (H1)} and {\rm (H2s)}, and let $\mu$ be an
$\eps$-minimizer in the sense of Definition {\rm \ref{def:min}} such that
$\supp(
\mu)$ is compact. Assume moreover that
$$
(-\Delta)^s W_a\in W^{\eps_0,1}_{loc}(\R^N)\quad \mbox{ for some small
$\eps_0>0$}.
$$ 
Then the potential function $\psi$ is in $\mathcal C^{1,\gamma}(\R^N)$ for all
$\gamma<s$. Furthermore, the measure $\mu$ is absolutely continuous with respect
to the Lebesgue measure and there exists a function $\rho \in C^{\alpha}(\R^N)$
for all $\alpha<1-s$, such that $\mu = \rho(x) d\mathcal L^N$.
\end{theorem}

\begin{remark}
The optimal regularity for the potential function in the fractional obstacle
problem is $C^{1,s}(\R^N)$, see {\rm \cite{CSS}}, but it requires $W_a*\mu\in
C^{2,1}(\R^N)$.
\end{remark}

Note that it has been conjectured that the free boundary has locally finite $N-1$
Hausdorff measure for the fractional obstacle problem, but to our knowledge
there is no proof of it yet.

\subsection{A uniqueness result}
We end this section with a uniqueness result for $d_2$-local
minimizers for the
very particular case 
of quadratic confinement where $W_a(x)=K |x|^2$.
This result gives an
alternative proof of the classical results by potential theory mentioned in the
introduction \cite{Fr,L-G,CGZ}, see also \cite{CFT}. 

Let us remark that the results in \cite{CCP} show the existence of global minimizers for
$W(x)=K|x|^2+V_s(x)$, see \cite[Section 3]{CCP}. Moreover, all global minimizers must be compactly supported and an attentive reading of Lemmas 2.6 and 2.7 in \cite{CCP} shows that any $d_2$-local minimizer is compactly supported in this particular case, since $W(x)\to\infty$ as $|x|\to \infty$.
Our result is:

\begin{theorem}[Uniqueness of $d_2$ minimizer]\label{them:fracunique}
Assume that $W_a(x)=K|x|^2$, where $K$ is a constant, and that {\rm(H1)} and either {\rm (H2)}  or {\rm(H2s)} hold. Then there exists a unique (up to translation) $d_2$-local minimizer $\mu_0\in \P_2(\R^N)$, which is
also the unique global minimizer of $E$ in $\P_2(\R^N)$. Furthermore, 
$\mu_0$ is compactly supported and radial symmetric. 
\end{theorem}

As discussed in Section 2,
for $d_2$-local minimizers, we have 
$$ 
\psi(x)=2E[\mu] \quad \mbox{ for all $x\in \supp(\mu)$}.
$$  
In other words, the constant $C_0$  appearing in Proposition \ref{prop:obstacle}
does not depend on the choice of $x_0$. 
We will thus prove Theorem \ref{them:fracunique} by using a uniqueness result for the obstacle problem, in the particular case treated in \cite{CV}.

In the case of Newtonian repulsion (condition {\rm(H2)}), we get a slightly better result. Indeed, 
for this choice of $W_a$,
we have $\Delta W_a =2NK$, so by
equation \eqref{eq:mupsil}, we have
\begin{equation}\label{eq:jnjn}
\mu(x)=F(x)=2NK \quad \mbox{ for all $x\in \supp(\mu)$}.
\end{equation}
In other words, the regularity that we have proven in the previous sections allows us to say that $\mu$
has to be a multiple of the characteristic function of its support, and it is not difficult to show that this support is a ball.

We can also use this information a priori and obtain a different proof of Theorem \ref{them:fracunique} for global minimizers in  the Newtonian case using a simple scaling argument and a trivial rearrangement.
This proof uses the obstacle problem to allow us to assume that $\mu$ is a function satisfying \eqref{eq:jnjn}.  We provide this proof for the interested reader at the end of Section \ref{sec:uniquenessproof}.


\section{Continuity of the potential function $\psi$}\label{sec:proofcont}
\setcounter{equation}{0}

This section contains the proof of Proposition \ref{prop:continuity}. First, we
recall that for a given $\mu\in \P(\R^N)$ and due to {\rm (H1)}, the potential
function $\psi = W*\mu$, is defined pointwise in $\R^N$, with values in
$[0,\infty]$, by
$$
\psi (x): =W*\mu= \int_{\R^N}  W(x-y)d\mu(y).
$$
Furthermore, $\psi$ is lower semicontinuous. We also recall that whether
{\rm (H3a)} or {\rm (H3b)} holds, we can always assume \eqref{eq:M}, i.e., that
there exists a constant $M$ such that  
\begin{equation*}
||\Delta W_a||_{L^p(B_1(x))}\leq M\quad  \mbox{  for all $x\in \R^N$.} 
\end{equation*}

The proof of Proposition \ref{prop:continuity} will rely on the following
classical mean value formula, whose proof is recalled in Appendix A for the  reader's
sake.

\begin{lemma}\label{lem:meanvalue}
Let $u$ be such that $\Delta u\in L^p_{loc}(\R^N)$ for some  $p\in
(\frac{N}{2},\infty]$.
Then for all $x$,  there exists a constant $C$ depending only on $N$ and $p$,
such that
\begin{equation*}
 u(x)\ge\frac{1}{|B_r|}\int_{B_r(x)}u(y)dy-C||\Delta
u||_{L^p(B_{1}(x))} r^\alpha, \quad \mbox{ for all $r\in(0,1)$}\,,
\end{equation*}
with $\alpha=2-\frac{N}{p}$, $N\geq 3$. When $N=2$, $r^\alpha$ is replaced by
$|\log r| r^\alpha$. 
\end{lemma}

As a simple application of this result, we can show

\begin{lemma}\label{lem:meanvalue2}
Let $\mu \in \P(\R^N)$ and assume {\rm (H1)}, {\rm (H2)}, and {\rm (H3)}. Then
there exists a constant $C$ depending on $N$, $p$, and the constant $M$
appearing in \eqref{eq:M}, such that the potential function $\psi=W*\mu$
satisfies
\begin{equation}\label{eq:supmeanvalue}
 \psi(x)\geq \frac{1}{|B_r|}\int_{B_r(x)} \psi(y)\, dy -Cr^\alpha
 \end{equation}
for all $x\in \R^N$ and all $r \in(0,1)$. When $N=2$, $r^\alpha$ is replaced by
$|\log r| r^\alpha$.
\end{lemma}

\begin{proof}[Proof of Lemma \ref{lem:meanvalue2}] Observe that we can always
assume $\psi(x)<\infty$, since otherwise \eqref{eq:supmeanvalue} clearly holds.
Next, we recall that $W=V+W_a$ where the function $V$ is super-harmonic, and
thus satisfies
$$ 
V(x-y )\geq \frac{1}{|B_r|}\int_{B_r(0)} V(x-y+z)\, dz
$$
for all $x$ and $y\in \R^N$. Furthermore, the function $W_a$ satisfies the
condition of Lemma \ref{lem:meanvalue} due to {\rm (H2)}, and so
$$ 
W(x-y)\geq \frac{1}{|B_r|}\int_{B_r(0)} W(x-y+z)\, dz -Cr^\alpha
$$
for all $r\in(0,1)$, where the constant $C$ depends on $M$. We deduce
$$ 
\psi(x) = \int_{\R^N} W(x-y)\, d\mu(y) \geq \frac{1}{|B_r|}  \int _{\R^N}
\int_{B_r(0)} W(x-y+z)\, dz\, d\mu(y) - Cr^\alpha\,,
$$
and Fubini-Tonelli theorem implies
\begin{align*}
\int _{\R^N} \int_{B_r(0)} W(x-y+z)\, dz \, d\mu(y) & = \int_{B_r(0)} \int
_{\R^N} W(x-y+z)\, d\mu(y)\, dz\\
& = \int_{B_r(x)}\psi(z) \, dz\,.
\end{align*}
Notice that the integral in the left hand side is finite thanks to {\rm (H1)}.
The result follows and the $N=2$ case is totally analogous.
\end{proof}

With the previous Lemma, we can prove the following Corollary that gives
important information about the potential function in the support of the local
minimizer.

\begin{corollary}\label{cor:minsupp}
Let $\mu$ be an $\eps$-minimizer of the energy $E$ in the sense of Definition
{\rm \ref{def:min}} and assume that {\rm (H1)}, {\rm (H2)}, and {\rm (H3)} hold.
Then any point $x_0\in\supp(\mu)$ is a local minimizer of $\psi=W*\mu$ in the
sense that
\begin{equation*}
 \psi(x_0) \leq \psi(x) \mbox{ for all } x\in B_\eps(x_0).
 \end{equation*}
Furthermore, we have
\begin{equation}\label{eq:psimincst}
\psi(x) = \psi(x_0) \mbox{ for all } x\in \supp(\mu)\cap B_\eps(x_0).
\end{equation}
\end{corollary}

\begin{proof}[Proof of Corollary \ref{cor:minsupp}]
Note that \eqref{eq:supmeanvalue} implies
\begin{equation}\label{eq:meanvalue} 
\psi(x) \ge \lim_{r\to 0 } \frac{1}{|B_r|}\int_{B_r(x)}\psi(y)\, dy \quad\mbox{
for all $x\in B_\eps(x_0)$.}
\end{equation}
But for any $x\in B_\epsilon(x_0)$, $B_r(x)\subset B_\epsilon(x_0)$ for small
enough $r>0$. Proposition~\ref{prop:minsupp} implies that
$$ 
\psi (y)\geq \psi(x_0)\quad\mbox{ a.e. } y\in B_r(x).
$$
So from \eqref{eq:meanvalue}, we conclude that
$$
\psi(x)\ge\psi(x_0) \quad \mbox{ for all $x\in B_\epsilon(x)$}.
$$
Because the $\eps$ does not depend on the choice of $x_0$ by Remark
\ref{rem:epsilon}, \eqref{eq:psimincst} easily follows.
\end{proof}

Using the previous Corollary we can prove (as a first step toward the continuity of
the potential function $\psi$) that $W_a*\mu$ is continuous.

\begin{lemma}\label{lem:W_a*mu}
 Let $\mu$ be an $\eps$-minimizer of the energy $E$ in the sense of Definition~{\rm \ref{def:min}} and assume that {\rm (H1)}, {\rm (H2)}, and either {\rm (H3a)} or {\rm (H3b)}
hold. Then, $\psi\in L^\infty_{loc}(\R^N)$ and $W_a*\mu\in C^\alpha_{loc} (\R^N)$, where
$\alpha=2-\frac{N}{p}$.
\end{lemma}
\begin{proof}
Condition {\rm (H2)} implies that $W_a\in C^\alpha_{loc}$, so under  {\rm (H3a)}, the result is obvious.
We can thus assume that only {\rm (H3b)} holds. The difficulty in that case is to deal with  the large values of $|x|$ in the support of $\mu$.

First, we prove that $\psi$ is finite everywhere in the
support of $\mu$. Clearly, we have $\psi(x)\geq 0$ (since $W\geq 0$).
Furthermore, given $y_0\in\supp(\mu)$, using Corollary~\ref{cor:minsupp} we know that
$\psi(x)\ge\psi(y_0)$ in $B_\eps(y_0)$ and we can write:
\begin{align*}
E[\mu] =\int_{\R^N} \psi(x)d\mu(x)\ge\int_{B_\eps(y_0)} \psi(x) d\mu(x) \ge \mu(B_\eps(y_0))\psi(y_0).
\end{align*}
Since $y_0\in\supp(\mu)$, we have $\mu(B_\eps(y_0))>0$ and so 
\begin{equation}\label{eq:supppsi}
\psi(y_0) \leq \frac{E[\mu]}{\mu(B_\eps(y_0))} <\infty.
\end{equation}

Now, given $x_0\in \R^N\setminus\supp(\mu)$, let $R=d(x_0,\supp(\mu))$, then by
{\rm
(H2)} we know that
$$ 
\Delta \psi  = \Delta W_a *\mu \geq -C_* \quad \mbox{ in }B_{R}(x_0)\,. 
$$
Therefore, by applying the mean-value formula to the subharmonic function
$\psi+C_*\frac{|x-x_0|^2}{2N}$, we obtain
$$ 
\psi(x_0)\leq \frac{1}{|B_{R}|}\int_{B_{R}(x_0)} \psi(y)\,
dy + \frac{C_*}{2N} R^2.
$$
Let now $y_0$ be the closest point in $\supp(\mu)$ to $x_0$.
Then proceeding as for Lemma~\ref{lem:meanvalue2}, we obtain
$$
\psi(y_0)\ge\frac{1}{|B_{2R}|}\int_{B_{2R}(x_1)}\psi(y)dy-C(R).
$$
Combining both inequalities we deduce
\begin{equation}\label{eq:psisupp}
\psi(x_0)\leq C(d(x_0,\supp(\mu)),\psi(y_0))<\infty\,.
\end{equation}

We now want to show that $\psi$ is bounded in $B_1(x_0)$.
First, assume that $\supp(\mu)\cap B_2(x_0)=\emptyset$. Then, for all $x\in B_1(x_0)$, we have 
$$ 
\Delta \psi  = \Delta W_a *\mu \geq -C_* \quad \mbox{ in }B_{1}(x) 
$$
and proceeding as above, we deduce
$$ \psi(x)\leq  \frac{1}{|B_{1}|}\int_{B_{1}(x)} \psi(y)\,
dy + \frac{C_*}{2N} \leq \frac{2}{|B_{2}|}\int_{B_{2}(x_0)} \psi(y)\,
dy + \frac{C_*}{2N} 
$$
and we conclude (using  Lemma~\ref{lem:meanvalue2}) that
\begin{equation}\label{eq:psinosupp} 
\psi(x)\leq C(\psi(x_0)) \quad \mbox{ in } B_1(x_0).
\end{equation}

If $\supp(\mu)\cap B_2(x_0)\neq \emptyset$, we note that $\supp(\mu)\cap B_2(x_0)$ can be 
covered by a finite number of balls of size $\eps$:
$$
\supp(\mu)\cap B_2(x_0)\subset \cup_{i=1}^p B_\eps (y_i) 
$$
with $\mu(B_\eps(y_i))>0$ for all $i$.
Since $\psi$ is constant in $\supp(\mu)\cap B_\eps(y_i)$, \eqref{eq:supppsi} implies that there exists a constant $C$ (depending on $x_0$ of course) such that
\begin{equation}\label{eq:psisuppb} 
\psi(y) \leq C \quad \mbox{ for all } y\in \supp(\mu)\cap B_2(x_0).
\end{equation}
Now, for a given $x\in B_1(x_0)$, either $B_1(x)\cap  \supp(\mu)=\emptyset$ in which case we get \eqref{eq:psinosupp}, or $d(x,\supp(\mu)) <1$, in which case  we conclude using \eqref{eq:psisupp} and \eqref{eq:psisuppb}.

This concludes the proof of the fact that $\psi$ is in $L^\infty_{loc}(\R^N)$.

\medskip

To prove the H\"older continuity of $W_a*\mu$, we take $x$ and $y$ such that $|x-y|\le
1$, then
$$
|W_a(x-z)-W_a(y-z)|\le ||W_a||_{C^\alpha(B_1(x-z))}|x-y|^\alpha
$$
for all $x,z \in \R^N$, and using Sobolev embeddings with $\alpha=2-\frac{N}{p}$, we get the bound
\begin{align*}
||W_a||_{C^\alpha(B_1(x-z))}
& \le
\int_{B_1(x)}|W_a(t-z)|\,dt+||\Delta W_a||_{L^p(B_1(x-z))}\\
&\le 
\int_{B_1(x)}|W_a(t-z)|\,dt+M
\end{align*}
for all $x,z \in \R^N$ (where M is the uniform bound from {\rm (H3b)}). 
We thus have
\begin{align*}
& \left| W_a*\mu(x)-W_a*\mu(y)\right| \\
& \qquad \le\int_{\R^N}|W_a(x-z)-W_a(y-z)|\,d\mu(z)&\\
& \qquad \le\int_{\R^N}\left(\int_{B_1(x)}|W_a(t-z)|\,dt+M \right)|x-y|^\alpha \,d\mu(z)&\\
&\qquad \le \left(\int_{B_1(x)}\int_{R^N}(W_a(t-z)-2\inf
W_a)\,d\mu(z)\,dt+M \right)|x-y|^\alpha.
\end{align*}
In dimension $N\geq 3$, using the fact that $V(x)\geq 0$, we deduce
\begin{align*}
& \left| W_a*\mu(x)-W_a*\mu(y)\right| \\
&\qquad \le \left(2C|\inf W_a|+\int_{B_1(x)} \psi(t)\,dt+M
\right)|x-y|^\alpha,
\end{align*}
and in dimension $N=2$, we get
\begin{align*}
& \left| W_a*\mu(x)-W_a*\mu(y)\right| \\
& \le \left(2C|\inf W_a|+\!\int_{B_1(x)} \!\!\!\!\!\psi(t)\,dt +C \!\int_{B_1(x)}\! \int_{\R^N} \!\!\!\log_+(|t-y|) d\mu(y)\,dt +M
\right)|x-y|^\alpha\\
& \le \left(2C|\inf W_a|+\int_{B_1(x)} \psi(t)\,dt +C \int_{\R^N} \log_+(1+|y|) d\mu(y) +M
\right)|x-y|^\alpha .
\end{align*}
We can now conclude using the fact that $\psi$ is in $L^\infty_{loc}$
(and \eqref{eq:logmu} when $N=2$).
\end{proof}

We now have everything that we need in order to prove
Proposition \ref{prop:continuity} by reproducing  Evans's proof for the
continuity of the
solution of the obstacle problem (see \cite{C} for instance):

\begin{proof}[Proof of  Proposition \ref{prop:continuity}]
Corollary \ref{cor:minsupp} implies that $\psi$ is continuous in the support of
$\mu$ in the following sense: If $x_k\in\supp(\mu)$ and $x_k\to
x_0\in\supp(\mu)$, then 
$\psi(x_k) \tto \psi (x_0)$.

Assume now that we have a sequence $x_k\in \R^N$ such that $x_k\to
x_0\notin\supp(\mu)$. By definition of $\supp(\mu)$, there exists a ball
$B_r(x_0)$ such that $\mu(B_r(x_0))=0$. 
In particular, the function $\psi-W_a*\mu$ is harmonic, and thus continuous, in $
 B_{r/2}(x_0)$. Furthermore, Lemma~\ref{lem:W_a*mu} implies that $W_a*\mu$ is continuous in $
 B_{r/2}(x_0)$, so $\psi$ is continuous in $ B_{r/2}(x_0)$, and $\lim_{k\to\infty} \psi(x_k)=\psi(x_0)$.

We now assume that $x_0\in\supp(\mu)$ and consider a sequence $x_k\in\R^N$ such
that $x_k\to x_0$ as $k\to\infty$. We can always assume that $x_k\in
B_\eps(x_0)$ for all $k$. In particular, if $x_k\in \supp(\mu)$ then
\eqref{eq:psimincst} implies $\psi(x_k)=\psi(x_0)$. So we can consider a
subsequence (still denoted $x_k$) such that 
$$
x_k\notin\supp(\mu) \mbox{ for all $k$,}\qquad \lim_{k\to\infty} x_k =
x_0\in\supp(\mu).
$$
This is of course the main step in Evans's Theorem. Since $x_k\in  B_\eps(x_0)$
for all $k$, 
\eqref{eq:psimin} implies
\begin{equation}\label{eq:limsup}
 \psi(x_k)\geq \psi(x_0) \mbox{ for all $k$}.
\end{equation}
Let $y_k$ be the closest point in $\supp(\mu)\cap \overline{B_{\eps/2}}(x_0)$ to
$x_k$. So \eqref{eq:psimincst} implies that $ \psi(y_k)=\psi(x_0)$. We denote
$\delta_k = |x_k-y_k|$. Notice that $\delta_k$ is  the distance of $x_k$ to
$\supp(\mu)$, and so $\delta_k \leq |x_k-x_0|\to 0$.
Inequality \eqref{eq:supmeanvalue} implies
$$
 \psi(x_0)  = \psi(y_k )  \geq
\frac{1}{|B_{2\delta_k}(y_k)|}\int_{B_{2\delta_k}(y_k)} \psi(y)\, dy
-C\delta_k^\alpha ,
$$
and so
$$  
\frac{1}{|B_{2\delta_k}(y_k)|}\int_{B_{2\delta_k}(y_k)} [\psi(y)-\psi(x_0)]\, dy
\leq C\delta_k^\alpha .
$$
Since $\psi(y)-\psi(x_0)\geq 0$ in $B_{2\delta_k}(y_k)$ (for $k$ large enough)
and $B_{\delta_k/2}(x_k)\subset B_{2\delta_k}(y_k)$, we deduce
\begin{align*}
 \frac{1}{4^d |B_{\delta_k/2}(x_k)|}\int_{B_{\delta_k/2}(x_k)}[
\psi(y)-\psi(x_0)]\,
dy \leq C\delta_k^\alpha .
\end{align*}
Furthermore, since $B_{\delta_k/2}(x_k)$ is away from $\supp(\mu)$, and $\Delta
V(x)=0$ for all $x\ne0$, we deduce by \eqref{eq:D1Wa} in {\rm (H2)} that
$$ 
\Delta \psi  = \Delta W_a *\mu \geq -C_* \quad \mbox{ in }B_{\delta_k/2}(x_k)
\,. 
$$
Since the function $\psi+C_* \frac{|x-x_k|^2}{2N}$ is subharmonic, we deduce
that
$$ 
\psi(x_k)\leq \frac{1}{|B_{\delta_k/2}|}\int_{B_{\delta_k/2}(x_k)} \psi(y)\,
dy + \frac{C_*}{2N} \delta_k^2
$$
and so
$$ 
\psi(x_k)-\psi(x_0) \leq
\frac{1}{|B_{\delta_k/2}|}\int_{B_{\delta_k/2}(x_k)}\!\!\!\!\!\!
[\psi(y)-\psi(x_0)]\, dy  + \frac{C_*}{2N}\delta_k^2 \leq 4^d \delta_k^\alpha
+ \frac{C_*}{2N}\delta_k^2.
$$
Together with \eqref{eq:limsup}, this implies
$$ 
\lim_{k\to\infty}\psi(x_k)=\psi(x_0)\,,
$$
and the result of Proposition \ref{prop:continuity} follows.
\end{proof}


\section{Regularity of $\mu$}
\setcounter{equation}{0}

In this section, we prove the regularity results for $\psi$ and $\mu$ stated in
Theorems \ref{thm:regularity} and \ref{thm:BV}. For that, we are going back and
forth between the regularity of the solution of the obstacle problem
\eqref{eq:obstacle2l} and the regularity of $F=\Delta W_a *\mu$.
To begin with, we note that Minkowski integral's inequality \cite[A.1]{St} and
{\rm(H3)} implies
\begin{align*} 
\left( \int_{B_1(x_0)} |F(x)|^p\, dx  \right)^{1/p}
 & = \left(\int_{B_1(x_0)} \left|\int_{\R^N}\Delta W_a (x-y)\, d\mu(y)
\right|^p\, dx \right)^{1/p} \\
 & \leq \int_{\R^N} \left( \int_{B_1(x_0)}| \Delta W_a (x-y)|^p\, dx
\right)^{1/p}\, d\mu(y) \\
 & \leq M\,,
\end{align*}
and so $F$ is locally uniformly in $L^p$.

\subsection{Proof of Theorem \ref{thm:regularity} (i)}
If we now assume that {\rm (H2)} holds with $p>N$. Then, for all $x_0\in
\supp(\mu)$, we have  $F\in L^p(B_\eps(x_0))$ with $p>N$, and  it is a classical
result \cite{Gu} that the solution of the obstacle problem \eqref{eq:obstacle2l}
is such that $\Delta \psi $ is a function and satisfies
$$ 
-F \leq -\Delta \psi \leq \max \{0,-F\} \qquad \mbox{ in } B_\eps(x_0).
$$
Since this is true for all $x_0\in \supp(\mu)$, we deduce that $\mu$ has density
$\rho$ and
\begin{equation}\label{eq:boundmu}  
0 \leq \rho \leq \max \{F,0\} \quad \mbox{ in } \R^N
\end{equation}
and so $\rho$ is locally uniformly in $L^p(\R^N)$. By assumption, we know that
$\rho$ is compactly supported, then $\rho\in L^p(\R^N)$ and thus, $F=\Delta
W_a\ast \rho \in C(\R^N)$ by standard convolution properties since  $p>N\geq 2$.
Using again \eqref{eq:boundmu} and the compact support of $\rho$, we deduce that
$\rho \in L^\infty(\R^N)$. This concludes the proof of Theorem
\ref{thm:regularity} when $(i)$ holds.

\subsection{Proof of Theorem \ref{thm:regularity} (ii)}
Unfortunately, when $p\leq N$ (which is the case for kernels of the form
${|x|^q}/{q}$ when $q\leq 1$), inequality  \eqref{eq:boundmu} does not hold, and
the argument above cannot be used. We are thus now going to prove the $L^\infty$
regularity of $\mu$ under Assumption $(ii)$ of Theorem \ref{thm:regularity}. 

Since we are assuming that $\supp (\mu) \subset B_R(0)$, and we are only
interested in the properties of $\psi$ and $\mu$ in a neighborhood of
$\supp(\mu)$, it is possible to modify the values of $W_a$ outside a ball
$B_{2R}(0)$ without changing the values of $\mu$ and $\psi $ in $B_R$ as
discussed in Section 3. We can thus assume that $\Delta W_a$ has compact support
and that
$$
\Delta W_a\in W^{\epsilon,1}(\R^N) \mbox{  for some $\epsilon>0$.}
$$
We will then use the following lemma (with $K=\Delta W_a$):

\begin{lemma}\label{lem:bootstrap}
If $K\in W^{\epsilon,1}(\R^N)$ with compact support and $\vphi\in
C^\alpha(\R^N)$ is a bounded function, then $K*\vphi \in C^{\alpha+\eps}$.
\end{lemma}
\begin{proof}[Proof of Lemma \ref{lem:bootstrap}]
We note that
$$ 
(-\Delta)^{\eps/2} (K*\vphi)  =  [(-\Delta)^{\eps/2} K]*\vphi
$$
is the convolution of an $L^1(\R^N)$ function with a $\mathcal C^\alpha(\R^N)$
function. We thus have 
$(-\Delta)^{\eps/2} (K*\vphi)  \in \mathcal C^\alpha(\R^N)$. Standard regularity
result for fractional elliptic equation (see \cite{S}) implies $K*\vphi \in
\mathcal C^{\alpha+\eps}(\R^N)$.
\end{proof}

We will then rely on the following important result for the regularity of the
solution of the obstacle problem in \cite{C}:
\begin{proposition}\label{prop:regularity}
Let $\psi $ be the solution of the obstacle problem \eqref{eq:obstacle2l}. Then
up to $C^{1,1}(\R^N)$ the function $ \psi$ is as regular as $W_a*\mu$.
More precisely, we have
\begin{itemize}
\item If $W_a*\mu$ has a modulus of continuity $\sigma(r)$, then $ \psi$ has a
modulus of continuity $C\sigma(2r)$.
\item If $\nabla W_a*\mu$ has a modulus of continuity of $\sigma(r)$, then
$\nabla
  \psi$ has a modulus of continuity $C\sigma(2r)$.
\end{itemize}
\end{proposition}

Using a bootstrap argument and Lemma \ref{lem:bootstrap}, we can now prove
Theorem \ref{thm:regularity} when $(ii)$ holds:

\begin{proof}[Proof of Theorem \ref{thm:regularity} (ii)]
{\rm (H2)} together with the compact support of $\Delta W_a$ imply that $\Delta
W_a*\mu \in L^p(\R^N)$ and so $W_a*\mu\in W^{2,p}(\R^N)$ for some $p>N/2$. In
particular, we have 
$$ 
W_a*\mu\in \mathcal C^\alpha(\R^N) \quad \mbox{ with } \alpha= 2 -\frac N p,
$$ 
and Proposition \ref{prop:regularity} implies
$$  
\psi \in \mathcal C^\alpha(\R^N) \quad \mbox{ with } \alpha= 2-\frac N p.
$$
Since $\mu = -\Delta   \psi+\Delta W_a*\mu$, we can write
$$ 
W_a*\mu = -W_a*\Delta   \psi +W_a*(\Delta W_a*\mu)=  -\Delta W_a*   (\psi
+W_a*\mu).
$$
As $\Delta W_a\in W^{\eps,1}(\R^N)$,  we can use Lemma \ref{lem:bootstrap} to
show that
$$
W_a*\mu \in  \mathcal C^{\alpha +\eps}(\R^N).
$$
A simple bootstrap argument yields that $W_a*\mu$ and $ \psi$ are both  in
$\mathcal C^{1,1}(\R^N)$, and thus that $\mu$ has density $\rho\in
L^\infty(\R^N)$.
\end{proof}


\subsection{$BV$ estimate and measure theoretic regularity of the free boundary}
We now prove Theorem \ref{thm:BV}.
We note that under the assumption of Theorem~\ref{thm:BV}, we have $F \in
W^{1,1}_{loc}(\R^N)$
and, for the second part of the statement, we get $F(x) \neq 0$ in a bounded
open neighborhood of $\supp(\mu)$. 

Under these assumptions, the regularity in $\BV_{loc}(\R^N)$ of $\Delta \psi$,
where $\psi$ solves the obstacle problem \eqref{eq:obstacle2l} is a classical
result, which implies Theorem \ref{thm:BV}.
We will sketch the  proof of this result for the reader's sake. The proof that
we give below was first proposed by Brezis and Kinderlehrer in \cite{BK}. 

\begin{proof}[Proof of Theorem \ref{thm:BV}]
First, we recall that the solution of the obstacle problem \eqref{eq:obstacle2l}
can be approximated
by the solutions $\psi_\delta$ of the nonlinear equation
\begin{equation}\label{eq:Deltadelta}
\begin{array}{rrr}
 -\Delta \psi_\delta +\beta_\delta(\psi_\delta-C_0)&=&-F\;\;\;in\;\Omega\\
  \psi_\delta&=&\psi\;\;\; on\;\partial\Omega
\end{array}
\end{equation}
where $\Omega=B_\epsilon(x_0)$ with $x_0\in\supp(\mu)$ and $\beta_\delta$ is an
increasing function satisfying $s\beta_\delta (s)\geq
0$ for all $s$ and such that 
$$
\beta_\delta(s)\tto 
\left\{\begin{array}{ll} 0 & \mbox{ when } s>0 \\ -\infty & \mbox{ when } s<0
\end{array}
\right.
\quad \mbox{ as $\delta\to 0$.}
$$ 
Here, $\Omega=B_\eps(x_0)$ for any point $x_0\in \supp(\mu)$. It is a classical
result, see \cite{BK} for instance, that $\psi_\delta$ converges to $\psi$
locally uniformly in $C^{1,\gamma}(\Omega)$ provided $F$ is in $L^\infty(\Omega)$ 
(which we proved in Theorem~\ref{thm:regularity}).

Let now $\frac{\pa}{\pa \xi}$ denote any directional derivative, we are going to
show that for any compact set $K\subset \subset  \Omega$, we have
\begin{equation}\label{eq:DeltaBV}
 \int_K \left|\frac{\partial}{\partial x_i}\Delta\psi_\delta\right|\le C.
\end{equation}
where $C$ does not depend on $\delta$. Taking the limit  $\delta\to0$ and using
the  l.s.c. of the total variation, we deduce  that $\Delta \psi\in
BV_{loc}(\Omega)$, which gives the result. 

In order to prove \eqref{eq:DeltaBV}, we differentiate \eqref{eq:Deltadelta}:
\begin{equation}\label{eq:Delta'}
  -\Delta\pa_\xi \psi_\delta +\beta_\delta'(\psi_\delta-C_0)\pa_\xi \psi_\delta
=-\pa_\xi F
  \end{equation}
Let now $\chi$ be a test function in $\mathcal D(\Omega)$ such that $\chi \geq
0$  in $\Omega$ and $\chi=1$ in $K$. We multiply \eqref{eq:Delta'} by $\chi
\sgn(\pa_\xi \psi_\delta )$ and integrate over $\Omega$ to deduce
$$ 
 -\int_\Omega \chi \sgn(\pa_\xi \psi_\delta ) \Delta \pa_\xi\psi_\delta \, dx 
+\int_\Omega \beta_\delta'(\psi_\delta-C_0)|\pa_\xi \psi_\delta | \chi \, dx
=
-\int_\Omega \pa_\xi F \chi \sgn(\pa_\xi \psi_\delta )\, dx\,.
$$
Integrating by parts the left hand side yields
\begin{align*}
&   \int_\Omega \chi \sgn'(\pa_\xi \psi_\delta ) |\na  \frac{\pa
}{\pa\xi}\psi_\delta|^2 \, dx  +\int_\Omega \beta_\delta'(\psi_\delta-C_0)|\pa_\xi
\psi_\delta | \chi \, dx \\
& \qquad =  -\int_\Omega \na \chi \na \pa_\xi \psi_\delta  \sgn(\pa_\xi
\psi_\delta ) \, dx-\int_\Omega \pa_\xi F \chi \sgn(\pa_\xi \psi_\delta )\,
dx\,.
\end{align*}
Using the fact that $\sgn'(s)\geq 0$ for all $s$, we deduce
\begin{align*} \int_\Omega \beta_\delta'(\psi_\delta-C_0)|\pa_\xi \psi_\delta |
\chi \, dx  & \leq 
-\int_\Omega \na \chi \na |\pa_\xi \psi_\delta | \, dx-\int_\Omega \pa_\xi F
\chi \sgn(\pa_\xi \psi_\delta )\, dx\\
& \leq 
\int_\Omega \Delta \chi  |\pa_\xi \psi_\delta | \, dx+ \int_\Omega | \pa_\xi F|
\chi \, dx\,.
\end{align*}
Furthermore, multiplying \eqref{eq:Deltadelta} by $(\psi_\delta -C_0) \chi$, it is
easy to show that
$$ 
\int_K |  \pa_\xi \psi_\delta |^2 \, dx \leq C(K)
$$
for some constant depending on $K$ but not on $\delta$ (using the
regularity of $F$ and the fact that $\psi_\delta$ converges locally uniformly to
$\psi$).  We conclude that 
$$  
\int_K \beta_\delta'(\psi_\delta-C_0)|\pa_\xi \psi_\delta | \, dx  \leq C
$$
with $C$ independent of $\delta$. Finally, going back to \eqref{eq:Delta'}, we
get
$$ 
\int_K |\Delta\pa_\xi \psi_\delta |\, dx \leq \int_K
\beta_\delta'(\psi_\delta-C_0)|\pa_\xi \psi_\delta| \, dx+ \int_K|\pa_\xi F|\,
dx\leq C
$$
and the result follows.

\

To prove the second part of the Theorem, we note that the function
$$
\frac{-\Delta\psi+F}{F} = \frac{\rho}{F}
$$
is almost everywhere equal to the indicator function of the set 
$\{\psi=\psi(x_0)\}\cap B_\eps(x_0)$.
If $F$ is never zero, we deduce that this function is in $BV_{loc}$, thus
proving that  $\{\psi=\psi(x_0)\}\cap B_\eps(x_0)$, and thus $ \supp(\mu)\cap
B_\eps(x_0)$ has finite perimeter.
Here, we use \eqref{eq:suppcontact} and the fact that if $E$ is a subset of $F$
and $|F\setminus E|=0$, then $E$ and $F$ have the same perimeter.
\end{proof}



\section{More singular than Newtonian repulsion}
\setcounter{equation}{0}

In this section, we consider kernel that are more repulsive than Newtonian. We
remind the reader that we assume that $W$ satisfies {\rm (H1)} and
{\rm (H2${}^s$)}. Moreover, we consider a  $\eps$-local minimizer $\mu$ (in the
sense of Definition \ref{def:min}) such that $\supp(\mu)$ is compact. We thus
have $\supp(\mu) \subset B_R(0)$ for some large $R$. As discussed in Section 3, we can cut off the values of $W_a(x)$ outside a large ball since those
values have no influence on the values of the potential $\psi$ in $B_R$, or on the
energy $E[\mu]$. Throughout this section, we can therefore assume
\begin{enumerate}
\item[(H3${}^s$)] The kernel $W_a$ is compactly supported, and there exists a
constant $M$ such that
\begin{equation}\label{eq:constantMs}
\|(-\Delta)^s W_a\|_{L^p(B_1(x))} \leq M \qquad \mbox { for all $x\in \R^N$}
\end{equation}
where $p\in(\frac{N}{2s},\infty]$.
\end{enumerate}

 \subsection{Continuity of $\psi$: Proof of Proposition \ref{prop:continuitys}}

A key tool in the proof of the continuity of $\psi$ in the Newtonian case was
the mean-value formula for Laplace's equation. Fortunately, this formula can be
generalized to the fractional Laplace equation (see \cite{S}). However, as usual
with fractional powers of the Laplacian, this formula is non local which
complicates the proof of Proposition \ref{prop:continuitys}. For that reason,
even though the main ideas are the same as in the proof of Proposition
\ref{prop:continuity}, we present here the proof of Proposition
\ref{prop:continuitys} in full details.

Our first task is to state this generalized mean-value formula. For that, we
recall that  $V_s(x)=\frac{c_{N,s}}{|x|^{N-2s}}$ is the fundamental solution of
$(-\Delta)^s$.
Proceeding as in \cite{S}, we define the function $\Gamma$ (we drop the index
$s$ for the sake of simplicity) as follows:
$$ 
\Gamma(x) = \left\{ \begin{array}{ll}
V_s(s) & \mbox{ in } |x|\geq 1 \\
A x^2 + B & \mbox{ in } |x|\leq 1 
\end{array}
\right.
$$
where the constant $A$ and $B$ are chosen in such a way that $\Gamma $ be a
$C^{1,1}$ function.
Given $\lambda>0$, we then scale $\Gamma$ in the following way:
$$
\Gamma_\lambda(x)=\frac{1}{\lambda^{N-2s}}\Gamma\left(\frac{x}{\lambda}\right).
$$ 
Note that the function $\Gamma_\lambda(x)$ coincides with $V_s(x)$ outside of
the ball of radius
$\lambda$. Furthermore, if $\lambda_1\le\lambda_2$, then
$\Gamma_{\lambda_1}\ge\Gamma_{\lambda_2}$.
Finally, we define
$$
\gamma_\lambda:= (-\Delta)^s\Gamma_\lambda.
$$
Note that $\gamma_\lambda(x) = \frac{1}{\lambda^N}\gamma_1(x/\lambda)$.
We then have the following result.

\begin{lemma}[\cite{S}]\label{lem:gamma}
\item[(i)] For all $\lambda>0$,  $\gamma_\lambda$ is a positive integrable
continuous function of unit mass. 
\item[(ii)] $x\mapsto \gamma_\lambda(x)$ decays like $\frac{1}{|x|^{N+2s}} $ as
$|x|\to \infty$.
\item[(iii)] The family of function $\{ \gamma_\lambda\}_{\lambda>0}$ is  an
approximation of the identity.
\end{lemma}

We can now state the mean-value formula for the fractional Laplace equation:

\begin{lemma}\label{lem:meanvaluefrac}
Assume that $u$ is an upper semi-continuous, compactly supported function and
$\Omega$ an open set in $\R^N$.
\item[(i)] If $(-\Delta)^s u\geq 0$ in $\Omega$, then for all $x\in \Omega$ and
for all $0<\lambda <d(x,\partial \Omega)$, it holds that
$$ 
u(x) \geq u*\gamma_{\lambda}(x)\,.
$$

\item [(ii)] If $(-\Delta)^s u\in L^p_{loc}(\R^N)$ for some  $p>\frac{N}{2s}$, 
then for all $x\in \R^N$ and  for all $\lambda\in(0,1)$, it holds that
\begin{equation*}
 u(x)\ge u*\gamma_\lambda(x)-C\|(-\Delta)^s u\|
_{L^p(B_1(x))}  \lambda^\alpha,
\end{equation*}
with $\alpha=2s-\frac{N}{p}$, where $C$ is a constant depending only on $N$, $s$
and $p$.
\end{lemma}

We recall the proof of this lemma in Appendix \ref{app:mvf} (for the sake of
completeness). 
Note that the assumption that $u$ be compactly supported is clearly not
necessary for such a formula to hold, but $u$ must satisfy some appropriate
integrability condition. For our purpose, it is enough to consider this simpler
case. We can now prove the continuity of the potential function.

\begin{lemma}\label{lem:meanvalue2frac}
Let $\mu \in \P(\R^N)$ compactly supported, and assume that {\rm (H1)}, {\rm
(H2${}^s$)} and {\rm (H3${^s}$)} hold. Then there exists a constant $C$
depending on  $N$, $s$, $p$ and the constant $M$ appearing in
\eqref{eq:constantMs} such that the potential function $\psi=W*\mu$ satisfies
\begin{equation}\label{eq:supmeanvaluefrac}
 \psi(x)\geq \psi*\gamma_\lambda (x) -C \lambda^\alpha
 \end{equation}
for all $x\in \R^N$ and for all $\lambda \in(0,1)$ and with $\alpha=2s-\frac{N}{p}$.
\item Furthermore,
\begin{equation}\label{eq:submeanvaluefrac}
 \psi(x)\leq \psi*\gamma_\lambda (x) +C_* \lambda ^{2s}
 \end{equation}
for all $x\in \R^N\setminus \supp(\mu)$ and for  all $0<\lambda<
d(x,\supp(\mu))$, where $C_*$ is the constant appearing in
\eqref{eq:fraclaplace}. 
\end{lemma}

\begin{proof}
We can always assume that $\psi(x)<\infty$ (since otherwise
\eqref{eq:supmeanvaluefrac} clearly holds). We have $W=V_s+W_a$, where
$V_s$ satisfies the conditions of Lemma~\ref{lem:meanvaluefrac}-(i)
and $W_a$ satisfies the conditions of Lemma~\ref{lem:meanvaluefrac}-(ii), so we
can write
$$ 
W(x-y)\geq W*\gamma_\lambda (x-y) -C\|(-\Delta)^s W_a\|_{L^p(B_1(x-y))}
\lambda^\alpha
$$
for all $x,\,y\in \R^N\times\R^N$ and for all $\lambda \in(0,1)$ where $C$ only
depends on $N$, $s$ and $p$. Using \eqref{eq:constantMs}, we deduce
$$ 
\psi(x) = \int_{\R^N} W(x-y)\, d\mu(y) \geq \int _{\R^N}
W*\gamma_\lambda (x-y)\, d\mu(y) - C(N,s,p,M) \lambda^\alpha
$$
and Fubini-Tonelli theorem implies
\begin{align*}
\int _{\R^N} W*\gamma_\lambda (x-y)\, d\mu(y)= \psi*\gamma_\lambda (x)\,.
\end{align*}
Note that the integral in the left hand side is finite due to the decay at
infinity of $V_s$. The first part of the result follows.

To prove the second part of the lemma, we fix $x\in\R^N\setminus\supp(\mu)$ and
 $0<\lambda<d(x,\supp(\mu))$ and we
consider $B=B_{r_0}(x)$ the biggest ball around $x$ which is contained in $\R^N\setminus
\supp{(\mu)}$ (so $r_0=d(x,\supp(\mu))$). 
We then define $v=C_*V_s*\chi_B$, which satisfies $(-\Delta)^s
v=C_*\chi_B$. 
Proceeding as in the proof of the mean
value formula (see Appendix~\ref{app:mvf}), we can write (integration by parts):
$$ \int_{\R^N} (-\Delta)^s v (x+y)(\Gamma_{\lambda_1}(y)-\Gamma_\lambda(y))\, dy = v * \gamma_{\lambda_1}(x)-v*\gamma_{\lambda}(x)$$
and so
$$ \int_{\R^N} C_*\chi_B  (x+y)(\Gamma_{\lambda_1}(y)-\Gamma_\lambda(y))\, dy = v * \gamma_{\lambda_1}(x)-v*\gamma_{\lambda}(x)$$
Letting $\lambda_1$ go to zero, we deduce
$$C_*  \int_{B_{r_0}(0)} (V_s(y)-\Gamma_\lambda(y))\, dy = v (x)-v*\gamma_{\lambda}(x)$$
Finally, the scaling of the integral in the left hand side (recall that $\Gamma_\lambda(x)=V_s(x)$ for $|x|\geq \lambda$) gives
$$v(x)-v*\gamma_\lambda(x) = C_* C\lambda^{2s}.$$ 
Hypothesis {\rm (H2${}^s$)} implies in particular that $(-\Delta)^s \psi\leq C_*$ in $\R^N\setminus
\supp{(\mu)}$ and so
$$(-\Delta)^s(\psi-v)\le0  \mbox{ in } B .$$ 
Using Lemma~\ref{lem:meanvaluefrac} (with $u=-(\psi-v)$), we get that $\psi(x)-v(x)\le
\psi*\gamma_\lambda(x)-v*\gamma_\lambda(x)$. Hence, 
$$
\psi(x)\le
\psi*\gamma_\lambda(x)+(v(x)-v*\gamma_\lambda(x))=\psi*\gamma_\lambda(x)+C_* C\lambda^{2s}
$$
which is exactly \eqref{eq:submeanvaluefrac}.
\end{proof}

\begin{lemma}\label{lem:below}
 Assume that $W$ satisfies {\rm (H1)}, {\rm (H2${}^s$)}, and {\rm (H3${^s}$)}.
Then the potential function $\psi=W*\mu$ is bounded below in $\R^N$.
\end{lemma}
\begin{proof}
 Clealy, we have $\psi-W_a*\mu=V_s*\mu\ge0$, so $\psi \geq W_a*\mu$. Furthermore
{\rm (H2${}^s$)} and {\rm (H3${}^s$)} implies that $W_a$ is continuous and
compactly supported, and thus bounded.
 The result follows.
\end{proof}

Using the previous lemmas we can now prove the following consequences.

\begin{corollary}\label{cor:minsuppfrac}
Assume that $W$ satisfies {\rm (H1)}, {\rm (H2${}^s$)}, and {\rm (H3${^s}$)}.
Let $\mu$  be a $\eps$-local minimizer in the sense of Definition {\rm
\ref{def:min}}.
Then any point $x_0\in\supp(\mu)$ is a local minimizer of $\psi=W*\mu$ in the
sense that
\begin{equation}\label{eq:psimin2frac}
 \psi(x_0) \leq \psi(x) \mbox{ for all } x\in B_\eps(x_0).
 \end{equation}
Furthermore, we have
\begin{equation}\label{eq:psimincstfrac}
\psi(x) = \psi(x_0)\qquad  \mbox{ for all } x\in \supp(\mu)\cap B_\eps(x_0).
\end{equation}
\end{corollary}

\begin{proof}
First, we note that \eqref{eq:supmeanvaluefrac} implies
\begin{equation*}
\psi(x) \ge \lim_{\lambda \to 0 } \psi*\gamma_\lambda (x) \quad\mbox{for all
$x$.}
\end{equation*}
Next, we recall that $\psi(x)\ge\psi(x_0)$ for a.e. $x\in B_\epsilon(x_0)$.
Using Lemma \ref{lem:below} and
\ref{lem:meanvalue2frac}, we can thus write, for $x\in B_{\eps}(x_0)$:
\begin{align*}
\psi*\gamma_\lambda (x) &= \int_{B_{\bar\eps}(0)} \psi(x-y)
\gamma_\lambda(y)\,dy + \int_{\R^N/B_{\bar\eps}(0)} \psi(x-y)
\gamma_\lambda(y)\,dy\\
&\geq \psi(x_0)A_1({\bar\eps},\lambda)+\left(\inf_{\R^N} \psi\right) A_2({\bar\eps},\lambda)
\end{align*}
for some $\bar\eps<\eps-|x-x_0|$ where
$$
A_1({\bar\eps},\lambda)=\int_{B_{\bar\eps}(0)} \gamma_\lambda(y)\,dy  \quad\mbox{ and
}\quad
A_2({\bar\eps},\lambda)=\int_{\R^N/B_{\bar\eps}(0)} \gamma_\lambda(y)\,dy\,.
$$
Lemma \ref{lem:gamma} implies 
$$ \lim_{\lambda\to 0} A_1({\bar\eps},\lambda) =1 , \qquad  \lim_{\lambda\to 0} A_2({\bar\eps},\lambda) =0,$$
which gives \eqref{eq:psimin2frac}.
The equality for $x\in \supp(\mu)$ follows from Remark \ref{rem:epsilon}.
\end{proof}

We can finally prove the first main result of this section: The continuity of the potential $\psi$. 
As for the Newtonian case, the proof follows essentially Evans proof for the continuity of the solution of the obstacle problem, which was adapted to the fractional case by L. Silvestre \cite{S}.

\begin{proof}[Proof of Proposition \ref{prop:continuitys}]
First, we consider a sequence $x_k\in \R^N$ such that $x_k\to
x_0\notin\supp(\mu)$.
By definition of $\supp(\mu)$, there exists a ball $B_r(x_0)$ such that
$\mu(B_r(x_0))=0$.
We can thus write
$$ 
\psi(x) = \int_{|x-y|\geq r/2} W(x-y)d\mu(y) \qquad \mbox{ for all } x\in
B_{r/2}(x_0)
$$
and we can always assume that $x_k\in B_{r/2}(x_0)$ (for $k$ large enough). 
Hypothesis {\rm (H2${^s}$)}  implies that $W(x)$ is continuous in
$\{x\in\R^N\,;\,
|x|\geq r/2\}$ and due to {\rm(H3${^s}$)} $W$ is compactly supported, so
$$
\lim_{k\to\infty } \psi(x_k)=\psi(x_0).
$$
This proves the continuity of $\psi$ outside of $\supp(\mu)$.
\medskip

We now fix $x_0\in\supp(\mu)$ and consider a sequence $x_k\in\R^N$ such
that $x_k\to x_0$ as $k\to\infty$. We can always assume that $x_k\in
B_\eps(x_0)$ for all $k$.
In particular, if $x_k\in \supp(\mu)$ then \eqref{eq:psimincstfrac} implies
$\psi(x_k)=\psi(x_0)$,
and so once again $\lim_{k\to\infty } \psi(x_k)=\psi(x_0)$.
So as in the Newtonian case, we are reduced to  considering a subsequence (still
denoted $x_k$) such that
$$ 
x_k\notin\supp(\mu) \mbox{ for all $k$,}\qquad \lim_{k\to\infty} x_k =
x_0\in\supp(\mu).
$$
Since $x_k\in  B_\eps(x_0)$ for all $k$, 
\eqref{eq:psimin2frac} implies
 \begin{equation}\label{eq:limsupfrac}
 \psi(x_k)\geq \psi(x_0) \mbox{ for all $k$}.
 \end{equation}
For all $k$ let $y_k$ be the closest point to $x_k$ which is in $\supp(\mu)\cap
\overline{B_{\eps/2}}(x_0)$. Again, for $k$ large enough, we can assume that
$y_k\in B_\eps(x_0)$ and so
\eqref{eq:psimincstfrac} implies that $ \psi(y_k)=\psi(x_0)$.
We denote $\delta_k = |x_k-y_k|$. Note that $\delta_k$ is  the distance
of $x_k$ to $\supp(\mu)$ for $k$ large enough, and so $\delta_k \leq
|x_k-x_0|\to 0$.
Inequality \eqref{eq:supmeanvaluefrac} implies
$$
 \psi(x_0)  = \psi(y_k )  \geq
\psi*\gamma_{\delta_k}(y_k)
-C\delta_k^\alpha.
$$

Let us define now  
$$
C_0=\inf_{x\in\R} \frac{\gamma_1(x+e)}{\gamma_1(x)}\,,
$$
where $e$ is any unit vector.
The infimum is achieved and it is positive, because
$\lim_{|x|\to\infty}\frac{\gamma_1(x+e)}{\gamma_1(x)}=1$. Furthermore, by
symmetry of
$\gamma$ it does not depend on the choice of $e$. Setting
$e=\frac{x_k-y_k}{\delta_k}$, we obtain
\begin{align}
\gamma_{\delta_k}(x-y_k)-C_0\gamma_{\delta_k}(x-x_k)&=\frac{1}{\delta_k^N}
\left(\gamma_{1}\left(\frac{x-x_k}{\delta_k}+\frac{x_k-y_k}{\delta_k}
\right)-C_0\gamma_{1}\left(\frac{x-x_k}{ \delta_k}\right)\right)\nonumber\\
&=\frac{1}{\delta_k^N}
\left(\gamma_{1}\left(\frac{x-x_k}{\delta_k}+e
\right)-C_0\gamma_{1}\left(\frac{x-x_k}{ \delta_k}\right)\right)\nonumber\\
&\ge 0. \label{eq:c0pos}
\end{align}
We now use \eqref{eq:supmeanvaluefrac} and \eqref{eq:submeanvaluefrac} to write,
for $k$ large enough,
\begin{align*}
\psi(x_0)=\psi(y_k)&\ge\psi*\gamma_{\delta_k}(y_k)-C\delta_k^\alpha\\
&=C_0
\psi*\gamma_{\delta_k}(x_k)+\int_{\R^N}(\gamma_{\delta_k}(y-y_k)-C_0\gamma_{
\delta_k}(y-x_k))\psi(y)\,dy-C\delta_k^\alpha\\
&\ge C_0\psi(x_k)-C_0C\delta_k^{2s}+I_1+I_2-C\delta_k^\alpha,
\end{align*}
where
$$
I_1=\int_{B_{\sqrt{\delta_k}}(y_k)}(\gamma_{\delta_k}(y-y_k)-C_0\gamma_{
\delta_k}(y-x_k))\psi(y)\,dy
$$
and
$$
I_2=\int_{\R^N\setminus
{B_{\sqrt{\delta_k}}}(y_k)}(\gamma_{\delta_k}(y-y_k)-C_0\gamma_{
\delta_k}(y-x_k))\psi(y)\,dy
$$

Using \eqref{eq:c0pos} and the fact that $\psi(y)\geq \psi(x_0)$ in
$B_{\sqrt{\delta_k}}(y_k)$, for $k$
large enough, we get
\begin{align*}
I_1 & \geq  \psi(x_0)
\int_{B_{\sqrt{\delta_k}}(y_k)}\big[\gamma_{\delta_k}(y-y_k)-C_0\gamma_{
\delta_k}(y-x_k)\big] \,dy\\
 & \geq \psi(x_0) \big[1-C_0-\eps_k\big]
\end{align*}
where
$$ 
\eps_k =  \int_{\R^N\setminus
B_{\sqrt{\delta_k}}(y_k)}\big[\gamma_{\delta_k}(y-y_k)-C_0\gamma_{
\delta_k}(y-x_k)\big] \,dy\,
$$
where we used the fact that $\gamma_\lambda$ has unit mass for all $\lambda$.
Making the change of variable  $z=\frac{y-y_k}{\delta_k}$ and using the notation
$e=\frac{y_k-x_k}{\delta_k}$, we find 
\begin{align*}
\eps_k =\int_{\R^N\setminus B_{\sqrt{\delta_k}}(0)}(\gamma_{1}
(z)-C_0\gamma_{1}(z+e))\,dz
\end{align*}
and so  
$$
\lim_{k\to \infty} \eps_k =0.
$$
Now we turn to $I_2$. Using \eqref{eq:c0pos} and Lemma \ref{lem:below},
we get
$$
I_2\ge \left(\inf_{\R^N} \psi\right)
\int_{\R^N\setminus
B_{\sqrt{\delta_k}}(y_k)}(\gamma_{\delta_k}(y-y_k)-C_0\gamma_{
\delta_k}(y-x_k))\,dy= \eps_k\inf_{\R^N} \psi.
$$

Combining all the above estimates, we conclude that, for $k$ large enough,
$$
\psi(x_0)=\psi(y_k)\ge C_0 \psi(x_k)+(1-C_0)
\psi(x_0)-C\delta_k^\alpha-C_0C_*\delta_k^{2s}-\eps_k \left(\psi(x_0)+
\inf_{\R^N} \psi\right).
$$
We deduce
$$
\limsup_{k\to \infty} \psi(x_k)\le\psi(x_0),
$$
which, together with \eqref{eq:limsupfrac}, implies
$$
\lim_{k\to \infty} \psi(x_k)=\psi(x_0)
$$
and completes the proof.
\end{proof}


\subsection{Regularity of the density function}

In order to apply known regularity results for the fractional obstacle problem
(as found, for instance, in \cite{S}), we need to show that $\psi$ solves  a fractional obstacle problem in the whole of $\R^N$. 

It is possible to do this  as follows:
The set $\supp(\mu)+B_{\eps/4} = \{x+y\, ;\, x\in \supp(\mu),\; y\in
B_{\eps/4}(0)\}$ is an open set in $B_{R+1}(0)$. In particular, it is the
countable union of its connected components $A_i$. 
Furthermore, since $\supp(\mu)$ is compact, there are only
finitely many $A_i$.

For all $i$, any two points $x_1$, $x_2$ in $\supp(\mu)\cap 
A_i$ will satisfy $\psi(x_1)=\psi(x_2)$, by the minimality of the connected
component and
Corollary~\ref{cor:minsuppfrac}. We define $D_i= \supp(\mu)\cap A_i$,
we denote $C_i = \psi|_{D_i}$ and we consider a smooth function $f$
such that $$f \leq  C_i \mbox{ in }  A_i$$
$$f = C_i \mbox{ on } D_i+B_\frac{\eps}{16}$$
$$ f = \inf W \mbox{ outside } \cup_i D_i+B_{\eps/8}.$$
We can find such smooth function, because $D_i$ are at least separated $\eps/4$
from each other, and if they are closer than $\eps$ then the constant $C_i$
has to match because of Corollary~\ref{cor:minsuppfrac}. 
The potential function
$\psi$ then solves the following obstacle problem in $\R^N$:
\begin{equation}\label{eq:fracobstaclen}
\left\{
\begin{array}{cl}
\psi \geq f,\quad-\Delta \psi  \geq  - F(x) & \quad \mbox{ in }\R^N \\
-(\Delta)^s\psi =  - F(x), & \quad \mbox{ in }  \{\vphi >f\} 
\end{array}
\right.
\end{equation} 
where $F=-(-\Delta)^s W_a*\mu$.

Using this obstacle problem formulation, we can use the following proposition which
is the fractional analog of Proposition \ref{prop:regularity} (See L. Silvestre \cite{S}):

\begin{proposition}\label{prop:fracregularity}
Let $\psi$ be the solution of the obstacle problem \eqref{eq:fracobstaclen}.
If $f\in C^2(\R^N)$ and  $W_a*\mu$ is in $C^{\beta}(\R^N)$ with $\beta>0$. Then $\psi\in
C^{\alpha}(\R^N)$
for every $\alpha< \min(\beta,1+s)$ (with the notation  $C^\alpha=C^{1,\alpha-1}$ if $\alpha> 1$).
\end{proposition}

\begin{proof}[Proof of Theorem \ref{thm:regularityfrac}]
We can now prove our main result by proceeding as in the proof of Theorem 
\ref{thm:regularity} (ii):

First, (H2${^s}$) implies that $(-\Delta)^s W_a*\mu \in L^p$ with $p>\frac{N}{2s}$ and so
$ W_a*\mu \in C^\alpha(\R^N) $ for $\alpha=2s-\frac{N}{p}$.
Proposition \ref{prop:fracregularity} implies that $\psi \in C^\alpha$.
Since $\mu = (-\Delta)^s   \psi-(-\Delta)^s W_a*\mu$, we can write
$$ 
W_a*\mu = W_a*(-\Delta)^s   \psi -W_a*((-\Delta)^s W_a*\mu)=  (-\Delta)^s W_a*   (\psi
+W_a*\mu).
$$
Using the fact that $(-\Delta)^s W_a\in W^{\eps,1}(\R^N)$ and Lemma \ref{lem:bootstrap}, this implies that
$$
W_a*\mu \in  \mathcal C^{\alpha +\eps}(\R^N).
$$

We iterate this argument until we get that $W_a*\mu\in C^{1,s}(\R^N)$, which then implies (by Proposition \ref{prop:fracregularity}) that  $\psi\in
C^{1,\gamma}(\R^N)$ for all $\gamma<s$. 
\medskip

Finally, this implies that
$\mu=(-\Delta)^s\psi-(-\Delta)^sW_a*\mu\in C^\beta(\R^N)$ for all $\beta<1-s$
(see Proposition 2.7 in \cite{S}) and complete the proof of Theorem \ref{thm:regularityfrac}.
\end{proof}


\section{Uniqueness: Proof of Theorem {\rm \ref{them:fracunique}}} \label{sec:uniquenessproof}


We now turn to the proof of our uniqueness result when $W_a=K|x|^2$. 
We assume that $\mu_0\in \P(\R^N)$ is  $d_2$ local minimizer of $E$, and we recall (see comment before Theorem  \ref{them:fracunique}) that such a minimizer is  compactly supported.
We will only prove the result for potential satisfying (H2s), since the Newtonian case is easier.

From our work in the previous section, we know that $\mu_0=\rho_0d\mathcal L^N$ where $\rho_0$ is a continuous function. Furthermore, we recall that for $d_2$-minimizers, the constant $C_i$ is the same on all connected component of $\supp(\mu_0)$ (and equal to $ 2E[\mu_0]$).
We deduce that $h=V_s*\mu_0$ satisfies
\begin{equation*}
\left\{
\begin{array}{rll}
h &\geq C-K|x|^2*\rho_0,\quad & \mbox{ in } \R^N \\
(-\Delta)^s h &\geq  0, \quad & \mbox{ in } \R^N \\
(-\Delta)^s h &=  0, & \mbox{ in }\{h >C-K|x|^2*\rho_0\} \\
\end{array}
\right.
\end{equation*}
for some constant $C$.

By translational invariance, we can always assume that the center of mass of
$\mu_0$ is zero,  and so
$$
K|x|^2*\rho_0=K\left[|x|^2+\int_{\R^N}
|y|^2\;\rho_0(y)\,dy\right]=K|x|^2
+C_{\rho_0}.
$$
Finally, using the fact that $\rho_0\in L^\infty(\R^N)$ and that $\rho_0$ has compact
support, it is easy to see that $h(x)\to 0$ as $|x|\to\infty$. 
Hence, $h$ solves the following  fractional obstacle
problem in $\R^N$:
\begin{equation}\label{eq:obstacleunique3}
\left\{
\begin{array}{rlc}
\vphi &\geq C-K|x|^2*\rho_0,\quad & \mbox{ in } \R^N \\
(-\Delta)^s \vphi &\geq  0, \quad & \mbox{ in } \R^N \\
(-\Delta)^s\vphi &=  0, & \mbox{ in }\{\vphi >C-K|x|^2*\rho_0\} \\
\vphi&\to0,& \mbox{when}\; x\to\infty,
\end{array}
\right.
\end{equation}
for some constant $C>0$. 

By Theorem 3.1 in \cite{CVobs}, we know that for any given $C>0$
\eqref{eq:obstacleunique3} has a unique solution $\vphi_C$. 
Furthermore, this solution is radially
symmetric and satisfies the scaling property:
$$
\vphi_C(x)=C\vphi_1\left(\frac{x}{C^{\frac{1}{2}}}\right).
$$ 
In particular, the function
$$
\rho_C(x):=(-\Delta)^s \vphi_C(x)=C^{1-s}\rho_1\left(\frac{x}{C^{\frac{1}{2}}}\right)
$$ 
satisfies
$$
\int_{\R^n}\rho_C(x)dx=C^{\frac{n}{2}+1-s}\int_{\R^n}\rho_1(x)dx.
$$
Therefore, there is a unique $C_0>0$ for which $\rho_{C_0}$ has unit
mass. 
It follows that $\rho_0=\rho_{C_0}$, which implies the  uniqueness of $\rho_0$ and the fact that it is radially symmetric.
\medskip

In the Newtonian case, one can then show that  the coincidence set is a
ball, use the maximum principle. The case involving the fractional Laplacian
is harder to characterize due to the non-locality of the problem.

\subsection{A different proof in the Newtonian case}\label{sec:unique}
In this section, we give an alternative proof of the uniqueness of global minimizer in the Newtonian case.
To simplify the notations, we assume that $W=V+\frac{|x|^2}{2N}$.

In particular, we have $F(x)=1$ and so minimizers of our energy (global or local) are of the form $\mu_0=
\chi_\Omega\, d\mathcal L^N$ with $|\Omega|=1$.
Using translation invariance, we can further assume that $\mu_0$ has zero center of mass.

Next, we use a simple scaling argument to derive an important relation:
Consider the dilation maps
$\mathcal{T}_\lambda:\R^N\to\R^N$ defined by $\mathcal{T}_\lambda(x)=\lambda x$
with $\lambda>0$. Given any $\mu\in\P(\R^N)$, we define
$\mu^\lambda=\mathcal{T}_\lambda\#\mu$. 
Using the fact that the (Newtonian) repulsive and (quadratic) attractive parts of the potential $W$ scales differently, we can re-write the energy of $\mu^\lambda$ in the following way:
\begin{equation}\label{eq:Escaling}
E[\mu^\lambda]=\frac{1}{\lambda^{N-2}}E_r[\mu]+\lambda^2 E_a[\mu],
\end{equation}
where $E_r[\mu]$ is the energy associated to the repulsive Newtonian interaction
potential $V$ and $E_a[\mu]$ the energy associated to the attractive quadratic
confinement ${|x|^2}/{2N}$.

For any $d_2$-local minimizer, the function
$\lambda\mapsto E[\mu^\lambda_0]$ must have a minimum at $\lambda=1$.
By
differentiating \eqref{eq:Escaling} we deduce
$$
\frac{dE[\mu^\lambda_0]}{d\lambda}\bigg{|}_{\lambda=1}=(2-N)E_r[\mu_0]+2E_a[
\mu_0] =0.
$$
Using this identity, we can rewrite the energy for $\mu_0$ as
\begin{equation}\label{energysets}
E[\mu_0]=\left(\frac{N-2}{2}+1\right)E_a[\mu_0]=\left(\frac{N-2}{2}
+1\right)\int_{\R^N} \frac{|x|^2}{N} d\mu_0(x).
\end{equation}

It remains to see that  this
implies that $\mu_0= \chi_{B_m}$ where $B_m$ is the ball centered at $0$ with
$|B_m|=1$.
This follows from the following two facts:
\begin{enumerate}
\item Among all sets
$\Omega$ of unit area and zero center of mass, the unique minimizer of 
\begin{equation*}
\int_\Omega |x|^2 dx
\end{equation*}
is the ball $B_m$.
\item The measure $\mu_0=\chi_{B_m}$ satisfies \eqref{energysets}.
\end{enumerate}

The second point is proved in \cite{CFT}, and it follows from the fact that $\chi_{B_m}$ is a $d_2$-local minimizer of $E$ under dilation. 

In order to prove the first point, we consider any measurable set $\Omega$ with zero center of mass and measure $1$, and we decompose it as
$\Omega=\Omega_a\cup\Omega_b$ with $\Omega_a=\Omega\cap B_m$ and
$\Omega_b=\Omega/\Omega_a$. Take the set $B_a=B_m/\Omega_a$. Since
$|B_m|=|\Omega|=1$, we must have $|\Omega_b|=|B_a|$. 
If $|B_a|=|\Omega_b|>0$, then we have
\begin{align*}
\int_{B_m} |x|^2 dx &=\int_{\Omega_a} |x|^2 dx +
\int_{B_a} |x|^2 dx\\
&< \int_{\Omega_a} |x|^2 dx + r_m^2 |B_a| = \int_{\Omega_a} |x|^2 dx + r_m^2
|\Omega_b|\\
& < \int_{\Omega_a} |x|^2 dx + \int_{\Omega_b} |x|^2 dx= \int_{\Omega} |x|^2
dx\,,
\end{align*}
where $r_m$ denotes the radius of $B_m$. 
This completes the proof of uniqueness of global minimizers. 

\medskip

Note that if $\mu_0$ is a $d_2$-local minimizer (rather than a global minimizer) with zero center of mass, we still have that
$\mu_0= \chi_\Omega\, d\mathcal L^N$ with $|\Omega|=1$, and \eqref{energysets} still holds.
Furthermore, if $\Omega$ is not the  ball $B_m$, we can show that by moving a  small amount of mass to a small ball in $B_m\setminus\Omega$, we obtain a set $\Omega'$
such that $\chi_{\Omega'}$ is close to $\chi_{\Omega} $ in the $d_2$ topology and 
$$\int_{\Omega'} |x|^2 dx< \int_\Omega |x|^2 dx$$
However, since $\chi_{\Omega'}$ might not satisfy \eqref{energysets}, we cannot get a contradiction. So this proof only applies to global minimizer, and we see that the fact that the measure $\mu_0=\chi_{B_m}$ satisfies \eqref{energysets} is really essential.

\appendix

\section{Mean value formula}
For the sake of completeness, we recall here the derivation of the mean-value formula (Lemma \ref{lem:meanvalue}).

We recall that $V(x)$ denotes the fundamental solution of the Laplace equation,
and we remove the singularity at $x=0$ by gluing a paraboloid to $V$ in the ball of
radius $1$. More precisely, we define a $C^{1,1}(\R^N)$ function $\Gamma$ by 
\begin{equation*}
\left\{
\begin{array}{lll}
\Gamma(x)&=V(x)&\mbox{in $|x|> 1$} \\
\Gamma(x)&=-\frac{1}{2|B_1|}|x|^2+C&\mbox{in $|x|\le 1$}
\end{array}
\right.
\end{equation*}
Given $\lambda>0$, also introduce
$\Gamma_\lambda=\frac{1}{\lambda^{N-2}}\Gamma(\frac{x}{\lambda})$. 
We note that the function
$\Gamma_\lambda$ coincides with $V(x)$ outside of the ball of radius
$\lambda$. Furthermore, if $\lambda_1\le\lambda_2$, then
$\Gamma_{\lambda_1}\ge\Gamma_{\lambda_2}$. Finally,  we note that
$-\Delta\Gamma_\lambda$ is an
approximation of the identity.  In fact, we have 
$-\Delta\Gamma_\lambda=\frac{1}{|B_\lambda|}\chi_{B_\lambda}$. Using these
facts, we can now prove Lemma \ref{lem:meanvalue}:

\begin{proof}[Proof of Lemma \ref{lem:meanvalue}]
Take $x\in B_R$ and $0<\lambda_1<\lambda_2<1$, then
$\Gamma_{\lambda_1}-\Gamma_{\lambda_2}$ is $C^{1,1}(\R^N)$ and compactly
supported
on the ball of radius $\lambda_2$. Using H\"older inequality we get
$$
\int_{\R^N}\Delta u(x+y)(\Gamma_{\lambda_1}(y)-\Gamma_{\lambda_2}(y))dy\le
||\Delta
u||_{L^p(B_{1}(x))}||\Gamma_{\lambda_1}-\Gamma_{\lambda_2}||_{L^q(B_{\lambda_2})
}.
$$
Integrating by parts the left hand side also gives
$$
\int_{\R^N}\Delta
u(x+y)(\Gamma_{\lambda_1}(y)-\Gamma_{\lambda_2}(y))dy=\frac{1}{|B_{\lambda_2}|}
\int_ { B_{\lambda_2}(x)}u(y)dy-\frac{1}{|B_{\lambda_1}|}
\int_ { B_{\lambda_1} (x)}u(y)dy
$$
and so
$$
\frac{1}{|B_{\lambda_1}|}\int_ { B_{\lambda_1}(x)}u(y)dy\ge
\frac{1}{|B_{\lambda_2}|}\int_ { B_{\lambda_2}(x)}u(y)dy-||\Delta
u||_{L^p(B_{R+1})}||\Gamma_{\lambda_1}-\Gamma_{\lambda_2}||_{L^q(B_{\lambda_2})}
. 
$$
Using the monotonicity of $\Gamma_\lambda$ with respect to $\lambda$, we have
$$
||\Gamma_{\lambda_1}-\Gamma_{\lambda_2}||_{L^q(B_{\lambda_2})}\le||\Gamma_{
\lambda_1}||_{L^q(B_{\lambda_2})}\le||V||_{L^q(B_{\lambda_2})}
$$
where a simple computation yields
$$
||V ||^q_{L^q(B_{\lambda_2})}=C\int_{B_{\lambda_2}}\frac{1}{|x|^{(N-2)q}} \, dx
=C\int_0^{\lambda_2}r^{N-1-(N-2)q}\, dr=C\lambda_2^{N-(N-2)q}
$$
when $N\geq 3$, and 
$$
||V ||^q_{L^q(B_{\lambda_2})}=C\int_{B_{\lambda_2}}|\log|x||^q\, dx 
=C\int_0^{\lambda_2}|\log r|^q r\, dr\leq |\log \lambda_2 |^q \lambda_2^2
$$
when $N=2$. We deduce that
$$
\frac{1}{|B_{\lambda_1}|}\int_ { B_{\lambda_1}(x)}u(y)dy\ge
\frac{1}{|B_{\lambda_2}|}\int_ { B_{\lambda_2}(x)}u(y)dy-C \lambda_2^\alpha 
$$
with $\alpha = 2-\frac N p$ when $N\geq 3$ (and the corresponding inequality
when $N=2$).
Finally, taking $\lambda_1\to 0$ using the continuity of $u$ (given by Sobolev's
embedding theorems), we obtain the desired result.
\end{proof}

\begin{remark}\label{rem:meanvalue*}
If we can assure that $u(x)\ge\lim_{r\to0^+}\frac{1}{|B_r|}\int_{B_r(x)}u(y)dy$
for every $x$, then we can replace the hypothesis of $\Delta u\in
L^p_{loc}(\R^N)$, by
$(\Delta u)_+\in L^p_{loc}(\R^N)$.
\end{remark}


\section{Fractional mean value formula}\label{app:mvf}

We now recall the derivation of the mean-value formula in the fractional case (see \cite{S}):

\begin{proof}[Proof of Lemma \ref{lem:meanvaluefrac}]
Take $0<\lambda_1<\lambda_2<d(x,\partial \Omega)$, then the function
$\Gamma_{\lambda_1}-\Gamma_{\lambda_2}$ is  a $C^{1,1}$ function,
which is non-negative  and compactly supported
in $\Omega$. As in the proof of the usual mean-value formula, we now consider
the integral
$$
\int_{\Omega}(-\Delta)^s u(x+y)(\Gamma_{\lambda_1}(y)-\Gamma_{\lambda_2}(y))dy
$$
An integration by parts yields
$$
\int_{\Omega}(-\Delta)^s
u(x+y)(\Gamma_{\lambda_1}(y)-\Gamma_{\lambda_2}(y))dy=u*\gamma_{\lambda_1}
(x)-u*\gamma_{\lambda_2}(x).
$$

If $(-\Delta)^s u\geq 0$ in $\Omega$, then the integrand in the left had side is
alway non-negative and we immediately deduce
$$ 
u*\gamma_{\lambda_1}(x)\geq u*\gamma_{\lambda_2}(x).
$$
Taking the limit $\lambda_1\to 0$ (using the upper semi-continuity of $u$), we
obtain
$$ 
u(x) \geq u*\gamma_{\lambda_2}(x)
$$
which is the first part of the lemma.

\medskip

If now $(-\Delta)^s u \in L^p_{loc}(\R^N)$, then we use H\"older inequality to
get
$$
\left|\int_{\R^N}(-\Delta)^s
u(x+y)(\Gamma_{\lambda_1}(y)-\Gamma_{\lambda_2}(y))dy\right|\le
||(-\Delta)^s
u||_{L^p(B_{1}(x))}||\Gamma_{\lambda_1}-\Gamma_{\lambda_2}||_{L^q(B_{\lambda_2}
(0))}.
$$
We deduce from the first part of the Lemma that
$$
u(x)\geq u*\gamma_{\lambda_1}
(x)\ge
u*\gamma_{\lambda_2}(x)-||(-\Delta)^s
u||_{L^p(B_{1}(x))}||\Gamma_{\lambda_1}-\Gamma_{\lambda_2}||_{L^q(B_{\lambda_2})
}
.
$$
Using the monotonicity of $\Gamma_\lambda$ with respect to  $\lambda$, we see
that $0\leq\Gamma_{\lambda_1}-\Gamma_{\lambda_2}\leq \Gamma_{\lambda_1}$ and so 
$$
||\Gamma_{\lambda_1}-\Gamma_{\lambda_2}||_{L^q(B_{\lambda_2})}\le||\Gamma_{
\lambda_1}||_{L^q(B_{\lambda_2})}\le|| V_s||_{L^q(B_{\lambda_2})}.
$$
Finally, a simple computation yields
$$
||V_s||^q_{L^q(B_{\lambda_2})}=c_{N,s}^q\int_{B_{\lambda_2}}\frac{1}{|x|^{
(N-2s)q}}
=C \int_0^{\lambda_2}r^{N-1-(N-2s)q}\, dr=C\lambda_2^{N-(N-2)q}
$$
for some constant $C$ depending only on $N$, $s$ and $q$.
We conclude that
$$
u(x)\ge u*\gamma_{\lambda_2}(x)-C ||(-\Delta)^s
u||_{L^p(B_{1}(x))} \lambda_2^\alpha\,,
$$
which is the desired result.
\end{proof}

\section*{Acknowledgement}
JAC acknowledges support from projects MTM2011-27739-C04-02, 2009-SGR-345 from Ag\`encia de Gesti\'o d'Ajuts Universitaris i de Recerca-Generalitat  de Catalunya, the Royal Society through a Wolfson Research Merit Award, and the Engineering and Physical Sciences Research Council (UK) grant number EP/K008404/1. AM was partially supported by NSF Grant DMS-1201426. The authors thank 
MSRI at Berkely where part of this work was carried over during the program on Optimal Transport.


\begin{thebibliography}{99}

\bibitem{Albietal}
G.~Albi, D.~Balagu\'e, J.~A. Carrillo, J.~von Brecht. Stability Analysis of
Flock and Mill rings for 2nd Order Models in Swarming. To appear in {\em SIAM J.
Appl. Math.}

\bibitem{BCLR}
D.~Balagu\'e, J.~A. Carrillo, T.~Laurent, and G.~Raoul.
Nonlocal interactions by repulsive-attractive potentials: Radial
ins/stability. {\em Physica D}, 260:5-25, 2013.

\bibitem{BCLR2}
D.~Balagu\'e, J.~A. Carrillo, T.~Laurent, and G.~Raoul.
 Dimensionality of local minimizers of the interaction energy.
 {\em Arch. Rat. Mech. Anal.}, 209(3):1055--1088, 2013.

\bibitem{BCY}
D.~Balagu{\'e}, J.~A. Carrillo, and Y.~Yao.
Confinement for repulsive-attractive kernels.
To appear in {\em DCDS-A}.

\bibitem{BT2}
A.~J. Bernoff and C.~M. Topaz.
\newblock A primer of swarm equilibria.
\newblock {\em SIAM J. Appl. Dyn. Syst.}, 10(1):212--250, 2011.

\bibitem{BCL}
A.~Bertozzi, J.~A. Carrillo, and T.~Laurent.
 Blowup in multidimensional aggregation equations with mildly singular
  interaction kernels.
 {\em Nonlinearity}, 22:683--710, 2009.

\bibitem{BLL}
A.~L. Bertozzi, T.~Laurent, and F.~L\'{e}ger.
 Aggregation and spreading via the newtonian potential: the dynamics
  of patch solutions.
 {\em Math. Models Methods Appl. Sci.},
  22(supp01):1140005, 2012.

 \bibitem{Blank}
 I.~Blank. Sharp results for the regularity and stability of the free
 boundary in the obstacle problem. {\em Indiana Univ. Math. J.} 50:1077--1112,
2001.

\bibitem{BK}
H.~Br\'ezis and D.~Kinderlehrer. The smoothness of solutions to
nonlinear variational inequalities. {\em Indiana Univ. Math. J.}, 23:831--844,
1974.

\bibitem{C}
L.~A.~Caffarelli. The obstacle problem revisited. {\em Journal of Fourier
Analysis and Applications}, 44: 383--402 (1998).

\bibitem{C1}
L.~A.~Caffarelli. A remark on the Hausdorff measure of a free boundary,
and the convergence of coincidence sets. {\em Boll. Un. Mat. Ital. A},
18.1:109-113, 1981.

\bibitem{CDMS} L.~A.~Caffarelli, J.~Dolbeault, P.~A.~Markowich, C:~Schmeiser. On
Maxwellian equilibria of insulated semiconductors. {\em Interfaces Free Bound.}
2:331–339, 2000. 

\bibitem{CF}
L.~A.~Caffarelli and A.~Friedman. A singular perturbation problem for
semiconductors.
{\em Boll. Un. Mat. Ital. B} 7.1:409–421, 1987.

\bibitem{CSS}
L.~A.~Caffarelli, S.~Salsa, and L.~Silvestre. Regularity estimates
for the solution and the free boundary of the obstacle problem for the
fractional Laplacian. {\em Inventiones mathematicae}, 171:425--461, 2008.

\bibitem{CV}
L.~A.~Caffarelli, J.~L.~V\'azquez. Nonlinear porous medium flow with fractional
potential pressure. {\em Arch. Ration. Mech. Anal.} 202:537–565, 2011. 

\bibitem{CVobs}
L.~A.~Caffarelli, J.~L.~V\'azquez. Asymptotic behaviour of a porous medium
equation with fractional diffusion. Discrete Contin. Dyn. Syst. 29:1393–1404,
2011. 

\bibitem{CCP}
J.~A. Ca\~nizo, J.~A. Carrillo, F.~S.~Patacchini.
Existence of Global Minimisers for the Interaction Energy. {\em preprint}, 2014.

\bibitem{CCH2} J.~A. Carrillo, A. Chertock, and Y. Huang. A Finite-Volume Method
for Nonlinear Nonlocal Equations with a Gradient Flow Structure. {\em preprint},
2014.

\bibitem{CCH} J.~A. Carrillo, M.~Chipot, and Y. Huang. On global minimizers of
repulsive-attractive power-law interaction energies. {\em preprint}, 2014.

\bibitem{CDFLS}
J.~A. Carrillo, M.~Di~Francesco, A.~Figalli, T.~Laurent, and D.~Slep\v{c}ev.
Global-in-time weak measure solutions and finite-time aggregation for 
nonlocal interaction equations. {\em Duke Math. J.}, 156:229--271, 2011.

\bibitem{CDFLS2}
J.~A. Carrillo, M.~Di~Francesco, A.~Figalli, T.~Laurent, and D.~Slep{\v{c}}ev.
Confinement in nonlocal interaction equations. {\em Nonlinear Anal.},
75(2):550--558, 2012.

\bibitem{CFP}
J.~A. Carrillo, L.~C.~F. Ferreira, J.~C. Precioso.
 A mass-transportation approach to a one dimensional fluid mechanics model with
nonlocal velocity.
 {\em Adv. Math.}, 231(1):306--327, 2012. 

\bibitem{CHM} J.~A. Carrillo, Y. Huang, S. Martin. Nonlinear stability of flock
solutions in second-order swarming models. {\em Nonlinear Analysis: Real World
Applications}, 17:332--343, 2014.

\bibitem{Carrillo-McCann-Villani03}
J.~A. Carrillo, R.~J. McCann, and C.~Villani.
 Kinetic equilibration rates for granular media and related equations:
 entropy dissipation and mass transportation estimates.
 {\em Rev. Mat. Iberoamericana}, 19(3):971--1018, 2003.

\bibitem{Carrillo-McCann-Villani06}
J.~A. Carrillo, R.~J. McCann, and C.~Villani.
 Contractions in the $2$-wasserstein length space and thermalization
  of granular media.
 {\em Arch. Rat. Mech. Anal.}, 179:217--263, 2006.

\bibitem{CGZ}
D. Chafa\"i, N.~Gozlan, P.-A.~Zitt. First order global asymptotics for confined
particles with singular pair repulsion. {\em preprint} arxiv:1304.7569v3, 2013.

\bibitem{CFT}
R. Choksi, R. Fetecau, I. Topaloglu. On Minimizers of Interaction Functionals
with Competing Attractive and Repulsive Potentials. Preprint 2013.

\bibitem{Dorsogna}
M.~R. D'Orsogna, Y.~Chuang, A.~Bertozzi, and L.~Chayes.
 Self-propelled particles with soft-core interactions: patterns,
  stability and collapse.
 {\em Phys. Rev. Lett.}, 96(104302), 2006.

\bibitem{Wales1995}
J.~P.~K. Doye, D.~J. Wales, and R.~S. Berry.
 The effect of the range of the potential on the structures of
  clusters.
 {\em J. Chem. Phys.}, 103:4234--4249, 1995.

\bibitem{FellnerRaoul1}
K.~Fellner and G.~Raoul.
 Stable stationary states of non-local interaction equations.
 {\em Math. Models Methods Appl. Sci.}, 20(12):2267--2291, 2010.

\bibitem{FellnerRaoul2}
K.~Fellner and G.~Raoul.
 Stability of stationary states of non-local equations with singular
  interaction potentials.
 {\em Math. Comput. Modelling}, 53(7-8):1436--1450, 2011.

\bibitem{FHK}
R.~C. Fetecau, Y.~Huang, and T.~Kolokolnikov.
 Swarm dynamics and equilibria for a nonlocal aggregation model.
 {\em Nonlinearity}, 24(10):2681--2716, 2011.

\bibitem{FH}
R.~C. Fetecau and Y.~Huang.
 Equilibria of biological aggregations with nonlocal repulsive--attractive
interactions.
 {\em Physica D}, 260:49--64, 2013.

\bibitem{Fr}
O.~Frostman. Potentiel d’Equilibre et Capacit\'e des Ensembles. {\em Ph.D.
thesis, Facult\'e des Sciences de Lund}, 1935.

\bibitem{GS}
C.~R. Givens and R.~M. Shortt.
 A class of {W}asserstein metrics for probability distributions.
 {\em Michigan Math. J.}, 31(2):231--240, 1984.

\bibitem{Gu}
B. Gustafsson. 
A simple proof of the regularity theorem for the variational inequality of the
obstacle problem.
{\em Nonlinear Anal. 10}, 12:1487--1490, 1986.

\bibitem{Viral_Capside}
M.~F. Hagan and D.~Chandler.
 Dynamic pathways for viral capsid assembly.
 {\em Biophysical Journal}, 91:42--54, 2006.

\bibitem{KCBFL}
T. Kolokolnikov, J. A. Carrillo, A. Bertozzi, R. Fetecau, M. Lewis.
 Emergent behaviour in multi-particle systems with non-local interactions,
 {\em Physica D: Nonlinear Phenomena}, 260:1--4, 2013.

\bibitem{KS}
D. Kinderlehrer and G. Stampacchia
\newblock {\em An introduction to variational inequalities and their
applications}, volume~88
  of {\em Pure and
Applied Mathematics}.
\newblock Academic Press, New York-London, 1980.


\bibitem{LT}
H.~Li and G.~Toscani. 
Long--time asymptotics of kinetic models of granular flows.
{\em Arch. Rat. Mech. Anal.}, 172(3):407--428, 2004.

\bibitem{L-G} A.~L\'opez-Garc\'ia. Greedy energy points with external fields.
{\em Recent trends in orthogonal polynomials and approximation theory, Contemp.
Math.}, 507:189--207, Amer. Math. Soc., Providence, RI, 2010.

\bibitem{Mati}
P.~Mattila.
\newblock {\em Geometry of sets and measures in {E}uclidean spaces}, volume~44
  of {\em Cambridge Studies in Advanced Mathematics}.
\newblock Cambridge University Press, Cambridge, 1995.
\newblock Fractals and rectifiability.

\bibitem{Mogilner2}
A.~Mogilner and L.~Edelstein-Keshet.
 A non-local model for a swarm.
 {\em J. Math. Bio.}, 38:534--570, 1999.

\bibitem{Mogilner2003}
A.~Mogilner, L.~Edelstein-Keshet, L.~Bent, and A.~Spiros.
 Mutual interactions, potentials, and individual distance in a social
  aggregation.
 {\em J. Math. Biol.}, 47(4):353--389, 2003.


\bibitem{Otto}
F. Otto. The geometry of dissipative evolution
equations: the porous medium equation. {\em Comm. Partial Differential
Equations}, 26:101--174, 2001.

\bibitem{PH} D.~Petz, F.~Hiai. Logarithmic energy as an entropy functional. {\em
Advances in differential equations and mathematical physics (Atlanta, GA, 1997),
Contemp. Math.}, 217:205--221, Amer. Math. Soc., Providence, RI, 1998.

\bibitem{Raoul}
G.~Raoul.
 Non-local interaction equations: Stationary states and stability
  analysis.
 {\em Differential Integral Equations}, 25(5-6):417--440, 2012.

\bibitem{Rechtsman2010}
M. C.~Rechtsman, F. H.~Stillinger, and S.~Torquato.
 Optimized interactions for targeted self-assembly: application to a
  honeycomb lattice.
 {\em Phys. Rev. Lett.}, 95(22), 2005.

\bibitem{HStable}
D.~Ruelle. {\em Statistical mechanics: {R}igorous results}. W. A. Benjamin,
Inc., New York-Amsterdam, 1969.

\bibitem{SV}
S. Serfaty, J. L. V\'azquez. Hydrodynamic Limit of Nonlinear Diffusions with Fractional Laplacian Operators. {\em Calc Var. PDE}, 2013.

\bibitem{S}
L.~Silvestre. Regularity of the obstacle problem for a fractional power
of the Laplace operator. {\em Communications on Pure and Applied Mathematics},
60:67--112, 2007.

\bibitem{SST}
R.~Simione, D.~Slep\v{c}ev, and I.~Topaloglu. Existence of minimizers of
nonlocal interaction
energies. Preprint, 2014.

\bibitem{St}
E.~Stein. Singular integrals and differentiability properties of functions.
Princeton University Press, 1970.

\bibitem{T}
F. Theil. A proof of crystallization in two dimensions. {\em Comm. Math. Phys.}
262:209–236, 2006. 

\bibitem{V}
C.~Villani.
 {\em Topics in optimal transportation}, volume~58 of {\em Graduate
  Studies in Mathematics}.
 American Mathematical Society, Providence, RI, 2003.

\bibitem{soccerball}
J.~von Brecht and D.~Uminsky.
 On soccer balls and linearized inverse statistical mechanics.
 {\em J. Nonlinear Sci.}, 22(6):935--959, 2012.

\bibitem{Wales2010}
D.~J. Wales.
 Energy landscapes of clusters bound by short-ranged potentials.
 {\em Chem. Eur. J. Chem. Phys.}, 11:2491--2494, 2010.

\end{thebibliography}
\end{document}